\documentclass[final]{siamart1116}

\usepackage{lipsum}
\usepackage{amsfonts}
\usepackage{graphicx}
\usepackage{subfigure}
\usepackage{epstopdf}
\usepackage{algorithmic}
\usepackage{booktabs}
\usepackage{tabularx}
\usepackage{multirow}
\usepackage{amsmath}
\usepackage{mathtools}
\usepackage{cite}
\usepackage{empheq}
\allowdisplaybreaks[4]
\ifpdf
  \DeclareGraphicsExtensions{.eps,.pdf,.png,.jpg}
\else
  \DeclareGraphicsExtensions{.eps}
\fi

\newsiamremark{remark}{Remark}
\newsiamremark{hypothesis}{Hypothesis}
\crefname{hypothesis}{Hypothesis}{Hypotheses}
\newsiamthm{claim}{Claim}
\newsiamremark{example}{Example}

\headers{Fast IMEX Time Integration of Nonlinear Stiff FDEs}{Y. Zhou, J.L. Suzuki, C. Zhang and M. Zayernouri}

\title{Fast IMEX Time Integration of Nonlinear Stiff Fractional Differential Equations\thanks{
\funding{This work is supported by NSFC (Grant No. 11971010), AFOSR YIP (award No. FA9550-17-1-0150), MURI/ARO (award No. W911NF-15-1-0562), ARO YIP (award No. W911NF-19-1-0444) and the NSF (award No. DMS-1923201). The work of the first author is supported by the China Scholarship Council under 201806160054.}}}

\author{Yongtao Zhou\thanks{School of Mathematics and Statistics, Huazhong University of Science and Technology, Wuhan 430074, China and Department of Mechanical Engineering, Michigan State University, East Lansing, MI 48824, USA
  (\email{yongtaozh@126.com}).}
\and Jorge L. Suzuki\thanks{Department of Mechanical Engineering, Michigan State University, East Lansing, MI 48824, USA
  (\email{suzukijo@msu.edu}, \email{suzukijo@egr.msu.edu}).}
\and Chengjian Zhang\thanks{School of Mathematics and Statistics, Huazhong University of Science and Technology, Wuhan 430074, China, and Hubei Key Laboratory of Engineering Modeling and Scientific Computing, Huazhong University of Science and Technology, Wuhan 430074, China
  (\email{cjzhang@mail.hust.edu.cn})}
  \and Mohsen Zayernouri\thanks{Department of Mechanical Engineering and Department of Statistics and Probability, Michigan State University, East Lansing, MI 48824, USA
 (\email{zayern@msu.edu}, \email{zayern@egr.msu.edu}), Corresponding Author.}}
 
 \usepackage{amsopn}

\begin{document}

\maketitle

\begin{abstract}
Efficient long-time integration of nonlinear fractional differential equations is significantly challenging due to the integro-differential nature of the fractional operators. In addition, the inherent non-smoothness introduced by the inverse power-law kernels deteriorates the accuracy and efficiency of many existing numerical methods. We develop two efficient first- and second-order implicit-explicit (IMEX) methods for accurate time-integration of stiff/nonlinear fractional differential equations with fractional order $\alpha \in (0,1]$ and prove their convergence and linear stability properties. The developed methods are based on a linear multi-step fractional Adams-Moulton method (FAMM), followed by the extrapolation of the nonlinear force terms. In order to handle the singularities nearby the initial time, we employ Lubich-like corrections to the resulting fractional operators. The obtained linear stability regions of the developed IMEX methods are larger than existing IMEX methods in the literature. Furthermore, the size of the stability regions increase with the decrease of fractional order values, which is suitable for stiff problems. We also rewrite the resulting IMEX methods in the language of nonlinear Toeplitz systems, where we employ a fast inversion scheme to achieve a computational complexity of $\mathcal{O}(N \log N)$, where $N$ denotes the number of time-steps.  Our computational results demonstrate that the developed schemes can achieve global first- and second-order accuracy for highly-oscillatory stiff/nonlinear problems with singularities. 
\end{abstract}



\section{Introduction}
\label{section1}

Fractional differential equations (FDEs) have been widely applied in a variety of scientific fields, where the observed data presents the trademark of power-laws/heavy-tailed statistics across many length/time scales. Some applications include, \textit{e.g.}, anomalous models for bio-tissues \cite{Magin2010,Naghibolhosseini2015,Naghibolhosseini2018}, food rheology \cite{Chen2013,Faber2017,Jaishankar2013} and earth sciences \cite{Zhangy2017}. Regarding nonlinear FDEs for anomalous transport/materials, we outline fractional Navier-Stokes equations \cite{Zhouy2017}, fractional phase-field equations \cite{Song2016}, complex constitutive laws applied to structural problems undergoing large deformations/strains \cite{Suzuki2016} as well as nonlinear vibrations of beams \cite{Varghaei2019}.

Obtaining closed forms for linear FDEs can be challenging, especially for any general form of $f(t)$. In the few instances when the corresponding solution $u(t)$ is known, it is usually impractical to be numerically evaluated. Furthermore, obtaining analytical solutions becomes impossible in the presence of nonlinearities. Therefore, a series of numerical methods for FDEs were developed since the 80's, with several significant contributions summarized in Figure~\ref{fig:timeline}. In such schemes it is fundamental to incorporate the history effects arising from the fractional operators. The pioneering works are attributed to Lubich \cite{Lubich1983,Lubich1986}, on \textit{discretized fractional calculus} in the sense of fractional multi-step and finite-difference (FD) schemes. Later on, Tang \cite{Tang1993} developed a super-linear convergent FD scheme, followed by a numerical quadrature approach for fractional derivatives introduced by Diethelm \cite{Diethelm1997}. In the 2000's, Diethelm developed a predictor-corrector approach in addition to a fractional Adams method \cite{Diethelm2002,Diethelm2004}. Later on, Lin and Xu \cite{Lin2007} developed a FD discretization with order $2-\alpha$, which was applied to the time-fractional diffusion equation. More recently, Garrappa \cite{Garrappa2015} developed trapezoidal methods for fractional multi-step approaches and Zeng \cite{Zeng2015} developed a second-order scheme for time-fractional diffusion equations. Spectral methods were also developed in the context of FDEs/FPDEs \cite{Zayernouri2014FracDelay,Zayernouri_FODEs_2014,zayernouri2015-Variable,kharazmi2017sem,Samiee2016,samiee2017unified,kharazmi2017FSEM,Rodrigues2016,Suzuki2016Spectral,zhou2019Prec}, and distributed-order differential equations \cite{kharazmi2017petrov,kharazmi2018fractional,samiee2018petrov}. In particular, Zayernouri and Karniadakis \cite{Zayernouri_FODEs_2014} developed an exponentially-accurate spectral element method for FDEs and Lischke \textit{et al.} \cite{Lischke2017} developed a fast, \textit{tunably-accurate} spectral method.

\begin{figure}[t]
    \centering
    \includegraphics[width=\columnwidth]{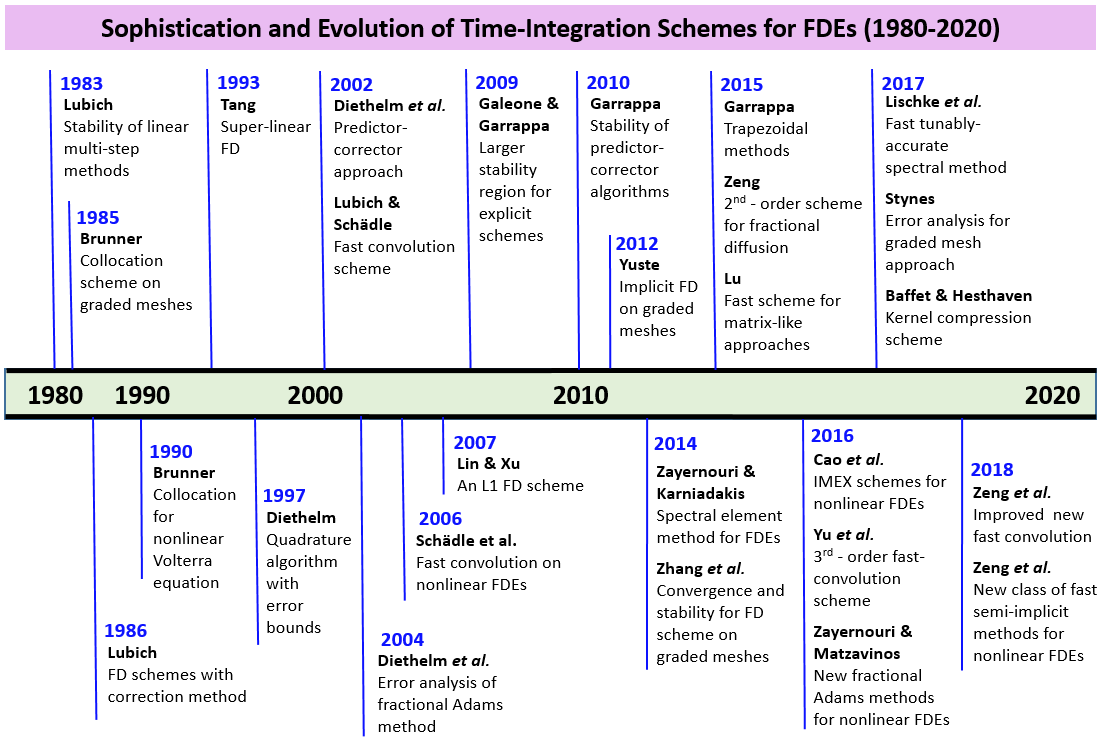}
    \caption{Research timeline on significant and diverse numerical schemes for time-fractional differential equations.\label{fig:timeline}}
\end{figure}

It is known that time-fractional operators possess power-law kernels with a singularity nearby the initial time, which produces non-smooth solutions that deteriorate the accuracy of many existing numerical schemes. In order to handle such problem, Lubich \cite{Lubich1986} introduced the \textit{so-called} correction method, which was later applied to a series of direct/multi-step schemes for linear/nonlinear FDEs \cite{Cao2016,Zeng2017CMAME,Zeng2018,Zeng2018stable}, and also employed in a \textit{self-singularity-capturing} approach by Suzuki and Zayernouri \cite{Suzuki2018SC}. In the aforementioned works, the correct determination of singularity powers leads to global high accuracy of the numerical schemes. The idea of \textit{graded meshes} was also introduced with the same objectives by Brunner \cite{Brunner1985}, who developed a spline collocation scheme for Volterra integro-differential equations, where the graded meshes correspond to non-uniform time-grids which are simple to incorporate in existing FD schemes. Graded meshes were later applied to nonlinear Volterra integral equations \cite{Brunner1990}. An implicit FD approach was developed in the context of graded meshes by Yuste \cite{Yuste2012}. More recently, stability issues of existing/new FD approaches were addressed by Zhang \textit{et al.} \cite{Zhang2014} for the time-fractional diffusion equation, and also by Stynes \cite{Stynes2017} for a reaction-diffusion problem, where, in the latter, an optimal mesh grading parameter was obtained. For a comparison between the performance of Lubich's corrections and graded meshes, we refer the readers to \cite{Zeng2018stable}.

The main computational challenge of direct FD schemes for time-integration of FDEs is the evaluation of the history term, which usually leads to a computational complexity of $\mathcal{O}(N^2)$ and memory storage of $\mathcal{O}(N)$, where $N$ represents the total number of time-steps. To address such issues, fast schemes were developed, starting with the first-order \textit{fast convolution method} by Lubich and Sch\"{a}dle \cite{Lubich2002}, which reduced the computational complexity to $\mathcal{O}(N \log N)$, and memory requirements to $\mathcal{O}(\log N)$. The main idea of the scheme is to approximate the power-law kernel \textit{via} numerical inverse Laplace transforms and split the integral operator (but not the time-grid) into exponentially increasing time steps. Later on, Sch\"{a}dle \textit{et al.} \cite{Schadle2006} extended the developed fast convolution for nonlinear FDEs. A third-order extension was developed by Yu \textit{et al.} \cite{Yu2016JCP} and applied to the time-integration for 3D simulation of a class of time-fractional PDEs. Zeng \textit{et al.} \cite{Zeng2018stable} developed an improved version of the fast-convolution approach, which considers real-valued integration contours with the order $3-\alpha$. Of particular interest, fast matrix-based schemes were also developed, such as the fast-inversion approach by Lu \cite{Lux2015} and the kernel compression method by Baffet and Hesthaven \cite{Baffet2017}. The main idea of such approaches is to represent the time-stepping equation in a global linear system and exploit the resulting Toeplitz-like structure through Fast-Fourier-Transforms (FFTs).

In addition to the aforementioned challenges, dealing with stiff/nonlinear problems further deteriorates the accuracy and might pose stability issues to existing numerical schemes. Fractional linear multi-step approaches become interesting alternatives to handle such issues. Diethelm \cite{Diethelm2002} developed a predictor-corrector scheme and later analyzed the error of a family of fractional Adams-Bashforth/Moulton schemes \cite{Diethelm2004}. Galeone and Garrappa studied the stability of implicit and explicit fractional multi-step methods \cite{Galeone2006,Galeone2009} and proposed new explicit schemes with larger stability regions. In addition, the stability analysis of fractional predictor-corrector schemes was studied by Garrappa \cite{Garrappa2010}. Also in the context fractional multi-step schemes, Zayernouri and Matzavinos \cite{Zayernouri2016MS} developed a family of fractional Adams schemes for high-order explicit/implicit treatment of nonlinear problems, where a particular time-splitting preserved the original structure of integer-order Adams schemes. Larger stability regions can be obtained through semi-implicit schemes, where for instance, Cao \textit{et al.} \cite{Cao2016} developed two IMEX schemes for nonlinear FDEs, utilizing two distinct force extrapolation formulas and also analyzed the stability of the developed schemes. Recently, Zeng \textit{et al.} \cite{Zeng2018} developed a new class of fast, second-order semi-implicit methods for nonlinear FDEs through new fast convolutions. Zhou and Zhang also developed and analyzed the convergence and stability of one-leg approaches \cite{zhou2019Stiff} and a class of boundary value methods and their block version \cite{zhou2019BBM,zhou2019BBVM} for stiff/nonlinear FDEs.

Although a significant amount of relevant works was developed, they usually address the aforementioned singularity/performance/stability issues for stiff/nonlinear problems separately. In this regard, there is still a need for numerical schemes in the context of stiff/nonlinear FDEs that \textbf{I)} efficiently handle the numerical solution with low-regularity for both the solution $u(t)$ and nonlinear term $f(t,\,u(t))$; \textbf{II)} present linear complexity with respect to the number of time-steps $N$; \textbf{III)} have larger stability regions compared to the existing numerical schemes; \textbf{IV)} mimick and generalize the structure of existing integer-order IMEX schemes, widely employed by the scientists and engineers to its fractional-order counterparts. The main contribution of the present work is to develop a class of IMEX methods for accurate time-integration of stiff/nonlinear FDEs. Specifically:
\begin{itemize}
    \item We start with the linear multi-step FAMMs developed by Zayernouri and Matzavinos \cite{Zayernouri2016MS} in the sense of linear/nonlinear fractional Cauchy equations. For the linear problem, two sets of Lubich-like correction terms \cite{Lubich1986} are employed.
    
    \item We develop a class of new first- and second-order IMEX methods with the combination of Zayernouri and Matzavinos \cite{Zayernouri2016MS} FAMMs with two extrapolation methods for the nonlinear term. The obtained methods are denoted by IMEX($p$), which is first-order accurate when $p=0$, and second-order accurate when $p=1$. 
    
    \item The convergence and linear stability of the developed IMEX methods are proved and the corresponding regions of stability are shown to be larger for smaller values of the fractional order $\alpha$.
    
    \item The convolution nature of the fractional operators allows us to represent the corresponding IMEX methods in the language of global-in-time Toeplitz-like nonlinear systems, and employ the fast approximate inversion approach by Lu \textit{et al.} \cite{Lux2015}. Since the Toeplitz system is nonlinear, we utilize a Picard iteration scheme which takes $N^p$ iterations until convergence with respect to a specified tolerance $\epsilon^p$. Under suitable conditions, $N^p$ does not significantly increase with $N$. 
    
    \item The corresponding \textit{history load} for the developed IMEX schemes is given by hypergeometric functions, which are efficiently evaluated through a Gauss-Jacobi quadrature with a fixed number of $Q$ integration points. 
    
    \item The asymptotic computational complexity of the scheme is $\mathcal{O}(N \log N)$, with memory storage of order $\mathcal{O}(N)$.
    
\end{itemize}

This paper is organized as follows: Section~\ref{section2} follows a step-by-step procedure, starting with linear multi-step FAMMs for linear FDEs, up to nonlinear FDEs, when the developed IMEX methods will be introduced. In Section~\ref{section3} we demonstrate the linear stability of the developed IMEX methods. In Section~\ref{section4} we put the corresponding IMEX methods in the language of a global nonlinear system of equations and employ a fast solver. The numerical results for linear/nonlinear/stiff FDEs with discussions are shown in Section~\ref{section5}, followed by the Conclusions in Section~\ref{section6}.

\section{Implicit-Explicit Time-Integration Methods}
\label{section2}

We develop two IMEX methods for efficient time-integration of nonlinear FDEs. In a step-by-step fashion, we start with the numerical solution of a fractional linear Cauchy problem, following the idea of FAMMs proposed by Zayernouri and Matzavions in \cite{Zayernouri2016MS}. To capture the singularity of the solution $u(t)$ of the considered problem, we then develop two sets of appropriate correction terms by using Lubich's approach \cite{Lubich1986} for the resulting fractional operators. As a next step, we introduce a nonlinear forcing term $f(t,u(t))$ and develop two IMEX methods for the solution of the resulting nonlinear Cauchy problem, which also introduces two additional sets of correction weights due to $f(t,u(t))$.

\subsection{Definitions}

We start with some preliminary definitions for fractional calculus (see e.g. \cite{Podlubny1999}). The left-sided Caputo fractional derivative of order $\alpha~(0 < \alpha < 1)$ is defined by
\begin{align}\label{2.1}
{}^C_{t_L} D_t^\alpha u(t) = {}_{t_L} I_t^{1-\alpha}u'(t) = \frac{1}{\Gamma(1-\alpha)} \int_{t_L}^t \frac{u'(v)}{(t-v)^{\alpha}} dv,~~~t > t_L,
\end{align}
where $\Gamma(\cdot)$ denotes the usual gamma function. The operator ${}_{t_L} I_t^{\alpha}$ represents the $\alpha$-th order $(0 <\alpha < 1)$ left-sided fractional Riemann-Liouville (RL) integral operator, defined as
\begin{align}\label{2.2}
{}_{t_L} I_t^\alpha u(t) = \frac{1}{\Gamma(\alpha)} \int_{t_L}^t \frac{u(v)}{(t-v)^{1-\alpha}} dv,~~~t > t_L.
\end{align}
The corresponding inverse operator of \eqref{2.2}, \textit{i.e.}, the left-sided Riemann-Liouville fractional derivative of order $\alpha$ is denoted by
\begin{align}\label{2.3}
{}^{RL}_{t_L} D_t^\alpha u(t) = \frac{d}{dt} \left[ {}_{t_L} I_t^{\alpha} u(t)\right] = \frac{1}{\Gamma(1-\alpha)} \frac{d}{dt} \int_{t_L}^t \frac{u(v)}{(t-v)^{\alpha}} dv,~~~t > t_L.
\end{align}
Moreover, the left-sided Caputo fractional derivative and the left-sided Riemann-Liouville fractional derivative are linked by the following relationship:
\begin{align}\label{2.4}
{}^{RL}_{t_L} D_t^\alpha u(t) = \frac{u(t_L)}{\Gamma(1-\alpha)(t-t_L)^{\alpha}} + {}^C_{t_L} D_t^\alpha u(t).
\end{align}

\subsection{Linear FDEs}
\label{sec:linear}

Consider the numerical solutions of the following linear FDE:
\begin{align}\label{2.5}
{}^C_{0} D_t^\alpha u(t) = \lambda u(t),~~\alpha \in (0,1],~t \in (0,T];~~~u(0)=u_0,
\end{align}
where $\lambda \in \mathbb{C}$, $u_0 \in \mathbb{R}^d$. Now, we adopt the FAMMs developed in \cite{Zayernouri2016MS} to solve \eqref{2.5}. Let $h>0$ be the time-step size with $h =  T/N~(N \in \mathbb{N})$ and $t_k = kh~(0\le k \le N)$. By the definition of the Caputo fractional derivative \eqref{2.1}, we can split the fractional operator in history and local parts:
\begin{align}\label{2.6}
{}^C_{0} D_t^\alpha u(t) = & \frac{1}{\Gamma(1-\alpha)} \int_{0}^{t_k} \frac{u'(v)}{(t-v)^{\alpha}} dv + \frac{1}{\Gamma(1-\alpha)} \int_{t_k}^{t} \frac{u'(v)}{(t-v)^{\alpha}} dv \nonumber\\
:= & H^k(t) + {}^C_{t_k} D_t^\alpha u(t).
\end{align}
Moreover, from \eqref{2.4}, we have
\begin{align}\label{2.7}
{}^C_{t_k} D_t^\alpha u(t) \!=\! {}^C_{t_k} D_t^\alpha [u(t) \!-\! u(t_k) \!+\! u(t_k)] \!=\! {}^C_{t_k} D_t^\alpha [u(t) \!-\! u(t_k)] \!=\! {}^{RL}_{t_k} D_t^\alpha [u(t) \!-\! u(t_k)].
\end{align}
Then, by substituting \eqref{2.6} and \eqref{2.7} into \eqref{2.5}, it holds that
\begin{align}\label{2.8}
{}^{RL}_{t_k} D_t^\alpha [u(t) - u(t_k)] = \lambda u(t) - H^k(t),~~~\alpha \in (0,1],~t \in (0,T].
\end{align}
Applying the inverse operator ${}_{t_k}I_t^\alpha$ on the both sides of \eqref{2.8} and evaluating at $t=t_{k+1}$, we obtain
\begin{align}\label{2.9}
u(t_{k+1})-u(t_k) = \left[ \lambda\, {}_{t_k}I_t^\alpha u(t) - \mathcal{H}^k (t)\right]\Big|_{t=t_{k+1}},
\end{align}
where $\mathcal{H}^k (t)$ denotes the \emph{history load term}, which is given by:
\begin{align}\label{2.10}
\mathcal{H}^k (t) = {}_{t_k}I_t^\alpha(H^k(t)) = \frac{1}{\Gamma(\alpha) \Gamma(1-\alpha)} \int_{t_k}^{t} \frac{1}{(t_{k+1}-v)^{1-\alpha}} \int_{0}^{t_k} \frac{u'(s)}{(v-s)^{\alpha}} ds dv.
\end{align}

We follow the FAMMs from \cite{Zayernouri2016MS} and interpolate $u(t)$ from ${}_{t_k}I_t^\alpha u(t)\Big|_{t=t_{k+1}}$ in an implicit fashion:
\begin{align}\label{2.11}
{}_{t_k}I_t^\alpha u(t) \Big|_{t=t_{k+1}} \approx h^\alpha \sum_{j=0}^p \beta_j^{(p)} u(t_{k+1-j}),~~~p=0,1,
\end{align}
with the following fractional Adams-Moulton coefficients, respectively, for $p=0$ and $p=1$,
\[\beta_0^{(0)} = \frac{1}{\Gamma(\alpha+1)};~~~\beta_0^{(1)} = \frac{1}{\Gamma(\alpha+2)},~~~\beta_1^{(1)} = \frac{\alpha}{\Gamma(\alpha+2)}.\]
Moreover, these coefficients recover the standard Adams-Moulton method's coefficients in the limit case when $\alpha=1$. To compute the history load term $\mathcal{H}^k (t_{k+1})$, on each small interval $[t_{j},t_{j+1}]~(0 \le j \le k-1)$, we linearly interpolate $u(t)$ when $p=0$, as follows:
\[\mit\Pi_{1,j} u(t) =\frac{t-t_{j+1}}{t_{j} - t_{j+1}}u(t_{j}) + \frac{t-t_j}{t_{j+1} - t_j}u(t_{j+1}),\]
therefore, we obtain the following form for the history load:
\begin{align}\label{2.12}
\mathcal{H}^k (t_{k+1}) \approx \mathcal{H}_{0}^{k} (t_{k+1}) = \frac{1}{\Gamma(\alpha) \Gamma(2-\alpha)} \sum_{j=0}^{k-1} \frac{u(t_{j+1}) - u(t_j)}{h} \left( \mathcal{A}_{k,j} - \mathcal{A}_{k,j+1}\right),
\end{align}
where
\begin{align}\label{2.13}
\mathcal{A}_{k,j} = & \int_{t_k}^{t_{k+1}} (t_{k+1}-v)^{\alpha-1} (v-t_j)^{1-\alpha} dv = h\int_{0}^{1} (1-\theta)^{\alpha-1} (k-j+\theta)^{1-\alpha} d\theta \nonumber\\
= & h\left\{ {\begin{array}{*{20}{l}}
\frac{(k-j)^{1-\alpha}}{\alpha}{}_2 F_1 \left(\alpha-1,1;\alpha+1;\frac{1}{j-k}\right),~~~&0 \le j < k,\\[3\jot]
\Gamma(\alpha)\Gamma(2-\alpha),&j=k,
\end{array}} \right.
\end{align}
in which ${}_2 F_1(a,b;c;z)$ denotes the hypergeometric function. Also, when $p=1$, we utilize quadratic interpolation function $\mit\Pi_{2,j} u(t)$ to approximate $u(t)$ on the interval $[t_j,t_{j+1}]~(0 \le j \le k-1)$ as follows:
\begin{align*}
\mit\Pi_{2,j} u(t) = &\frac{(t-t_{j})(t-t_{j+1})}{(t_{j-1}-t_{j})(t_{j-1}-t_{j+1})}u(t_{j-1}) + \frac{(t-t_{j-1})(t-t_{j+1})}{(t_{j}-t_{j-1})(t_{j}-t_{j+1})}u(t_{j}) \\
&+ \frac{(t-t_{j-1})(t-t_{j})}{(t_{j+1}-t_{j-1})(t_{j+1}-t_{j})}u(t_{j+1}),
\end{align*}
and therefore, the history load for the choice of $p=1$ is given by
\begin{align}\label{2.14}
\mathcal{H}^k & (t_{k+1}) \approx \mathcal{H}_{1}^k (t_{k+1}) = \mathcal{H}_{0}^k (t_{k+1}) \nonumber\\
& \!+\! \frac{1}{\Gamma(\alpha) \Gamma(2\!-\!\alpha)} \!\sum_{j=1}^{k-1} \!\frac{u(t_{j\!+\!1}) \!-\! 2u(t_j) \!+\! u(t_{j\!-\!1})}{h^2} \!\left[\! -\frac{h(\mathcal{A}_{k,j} \!+\! \mathcal{A}_{k,j\!+\!1})}{2} \!+\! \frac{\mathcal{B}_{k,j} \!-\! \mathcal{B}_{k,j\!+\!1}}{2-\alpha}\!\right],
\end{align}
where
\begin{align}\label{2.15}
\mathcal{B}_{k,j} = & \int_{t_k}^{t_{k+1}} (t_{k+1}-v)^{\alpha-1} (v-t_j)^{2-\alpha} dv = h^2\int_{0}^{1} (1-\theta)^{\alpha-1} (k-j+\theta)^{2-\alpha} d\theta \nonumber\\
= & h^2 \left\{ {\begin{array}{*{20}{l}}
\frac{(k-j)^{2-\alpha}}{\alpha}{}_2 F_1 \left(\alpha-2,1;\alpha+1;\frac{1}{j-k}\right),~~~&0 \le j < k,\\[3\jot]
\frac{\Gamma(\alpha)\Gamma(3-\alpha)}{2},&j=k.
\end{array}} \right.
\end{align}

Let $u_k$ be the approximate solution of $u(t_k)~(0 \le k \le N)$, and denote
\begin{align}\label{2.15*}
\mathcal{H}_p^{k} = \frac{1}{\Gamma(\alpha) \Gamma(2-\alpha)} \sum_{j=0}^{k} \gamma_{k,j}^{(p)} u_j,~~~p=0,1,
\end{align}
where the coefficients $\gamma_{k,j}^{(p)}$ are presented in Appendix \ref{appendixA}.
Then, we get the FAMMs for \eqref{2.9} have the following discrete form:
\begin{align}\label{2.16}
\frac{u_{k+1} - u_k}{h^\alpha} = \lambda \sum_{j=0}^p \beta_j^{(p)} u_{k+1-j} - \frac{1}{h^\alpha}\mathcal{H}_p^{k},~~~p=0, 1.
\end{align}
In order to lay the analytical basis for the convergence and linear stability analysis of the methods, we introduce several preparatory results through a series of lemmas, with their corresponding proofs given in Appendix \ref{appendix}.

\begin{lemma}\label{lemma2.3}
The coefficients $\big\{ \gamma_{k,j}^{(p)} \big\}$ in \eqref{2.15*} have the following properties for $0 \le j \le k$, $0 \le k \le N$ and $p=0,1$:
\begin{itemize}
\item $\gamma_{k,j}^{(0)} < 0~(0 \le j \le k-1),~0<\gamma_{k,k}^{(0)} \le C_1$;
\vskip 0.1 cm
\item $\gamma_{k,j}^{(1)} < 0~(0 \le j \le k-3),~\left|\gamma_{k,j}^{(1)}\right| \le C_1~(j=k-2,~k-1~\mbox{or}~k)$;
\vskip 0.1 cm
\item $\sum\limits_{j=0}^k \gamma_{k,j}^{(0)} = 0,~\sum\limits_{j=0}^k \gamma_{k,j}^{(1)} = 0$,
\end{itemize}
where $C_1 > 0$ is a constant.
\end{lemma}

\begin{lemma}\label{lemma2.1}
Let $u(t) = t^\sigma~(\sigma \ge 0)$. Then there exists a constant $C_2 > 0$ independent of $h$ such that
\begin{align*}
\bigg| {}_{t_k}I_t^\alpha u(t) \Big|_{t=t_{k+1}} - h^\alpha \sum_{j=0}^p \beta_j^{(p)} u(t_{k+1-j}) \bigg| \le C_2 h^{\alpha+p+1} t_{k+1}^{\sigma-p-1}.
\end{align*}
\end{lemma}

\begin{lemma}\label{lemma2.2}
Let $u(t) = t^\sigma~(\sigma \ge 0)$. Then there exists a constant $C_3 > 0$ independent of $h$ such that
\begin{align*}
\left|\mathcal{H}^k (t_{k+1}) - \mathcal{H}_{p}^k (t_{k+1})\right| \le C_3 \left(h^{\sigma+\alpha+1} t_{k + 1}^{-\alpha-1} + h^{\alpha+p+1} t_{k+1}^{\sigma-\alpha-p-1} + h^{p+1} \right).
\end{align*}
\end{lemma}

\subsection{Correction Terms}

It is well-known that the solutions of \eqref{2.5} usually exhibit weak singularity at the initial time. Hence, the optimal convergence rates of the above discussed numerical methods cannot be achieved (see Lemma \ref{lemma2.1} and \ref{lemma2.2}). To improve the accuracy near the initial time, we follow Lubich's idea (cf. \cite{Lubich1986}) by adding correction terms to the resulting fractional operators of \eqref{2.9}. The correction scheme assumes that the solution $u(t)$ of \eqref{2.5} has the form (\textit{see e.g. \cite{Cao2016,Diethelm2004,Zeng2017CMAME} for more discussions on the regularity of FDEs}):
\begin{align}\label{2.17}
u(t) - u(0) = \sum_{r=1}^{m+1} d_r t^{\sigma_r} + \zeta(t) t^{\sigma_{m +2}},~~~0 < \sigma_r < \sigma_{r+1},
\end{align}
where $d_r \in \mathbb{R}$ are some constants, $m$ is a positive integer and $\zeta(t)$ is a uniformly continuous function for $t \in [0,T]$. The term $\left\{\sigma_r \right\}$ represents positive correction powers. We will now introduce the correction approach for each term of the right-hand-side of \eqref{2.9}.

\textbf{I)} We start with the term ${}_{t_k}I_t^\alpha u(t) \Big|_{t=t_{k+1}}$. It follows from Lemma \ref{lemma2.1} that
\begin{align}\label{2.18}
{}_{t_k}I_t^\alpha u(t) \Big|_{t=t_{k+1}} = & {}_{t_k}I_t^\alpha \left(u(t) - u(0)\right) \Big|_{t=t_{k+1}} + {}_{t_k}I_t^\alpha u(0) \Big|_{t=t_{k+1}} \nonumber\\
= & h^\alpha \sum_{j=0}^p \beta_j^{(p)} \left(u(t_{k+1-j}) - u_0\right) + h^\alpha \sum_{j=1}^{m_u} W_{k,j}^{(\alpha,\sigma,p)} \left(u(t_{j}) - u_0\right) \nonumber\\
& + \frac{h^\alpha}{\Gamma(\alpha + 1)} u_0 + R_{k+1}^u \nonumber\\
= &  h^\alpha \sum_{j=0}^p \beta_j^{(p)} u(t_{k+1-j}) + h^\alpha \sum_{j=1}^{m_u} W_{k,j}^{(\alpha,\sigma,p)} \left(u(t_{j}) - u_0\right) + R_{k+1}^u \nonumber\\
 := & h^\alpha I_{h,m_u}^{\alpha,\sigma,p} u(t_{k+1}) + R_{k+1}^u,
\end{align}
with $R_{k+1}^u = \mathcal{O}(h^{\alpha+p+1} t_{k+1}^{\sigma_{m_u+1}-p-1})$. The correction weights $\left\{W_{k,j}^{(\alpha,\sigma,p)}\right\}$ are chosen such that \eqref{2.18} is exact for $u(t) = t^{\sigma_r}~(1 \le r \le m_u)$, and therefore are obtained through the following linear system of order $m_u \times m_u$ to be solved for $k = 0,1,\ldots,N-1$ time steps:
\begin{align}\label{2.19}
\sum_{j=1}^{m_u} W_{0,j}^{(\alpha,\sigma,p)} j^{\sigma_r} = \frac{\Gamma(\sigma_r+1)}{\Gamma(\sigma_r+\alpha+1)} - \sum_{j=0}^p \beta_j^{(p)} (1-j)^{\sigma_r},~~~1 \le r \le m_u,
\end{align}
\begin{align}\label{2.20}
\sum_{j=1}^{m_u} W_{k,j}^{(\alpha,\sigma,p)} j^{\sigma_r} =& \frac{k^{\sigma_r}}{\Gamma(\alpha+1)} {}_2F_1\left(-\sigma_r,1;\alpha+1;-\frac{1}{k}\right) - \sum_{j=0}^p \beta_j^{(p)} (k+1-j)^{\sigma_r}, \nonumber\\
&1 \le k \le N-1,~~~1 \le r \le m_u.
\end{align}

\textbf{II)} For the history load $\mathcal{H}^k (t_{k+1})$, we introduce the correction terms as follows. Substituting \eqref{2.17} into \eqref{2.10} and using Lemma \ref{lemma2.2} yield
\begin{align}\label{2.21}
\mathcal{H}^k (t_{k+1}) = & \frac{1}{\Gamma(\alpha) \Gamma(1-\alpha)} \int_{t_k}^{t_{k+1}} \frac{1}{(t_{k+1}-v)^{1-\alpha}} \int_{0}^{t_k} \frac{\left[u(s) - u(0)\right]'}{(v-s)^{\alpha}} ds dv \nonumber\\
& + \frac{1}{\Gamma(\alpha) \Gamma(1-\alpha)} \int_{t_k}^{t_{k+1}} \frac{1}{(t_{k+1}-v)^{1-\alpha}} \int_{0}^{t_k} \frac{u'(0)}{(v-s)^{\alpha}} ds dv \nonumber\\
= &\frac{1}{\Gamma(\alpha) \Gamma(2-\alpha)} \sum_{j=0}^{k} \gamma_{k,j}^{(p)} u(t_j) + \sum_{j=1}^{\tilde{m}_u} \tilde{W}_{k,j}^{(\alpha,\sigma,p)} \left(u(t_{j}) - u_0\right) + \tilde{R}_{k+1}^u \nonumber\\
:= & \mathcal{H}_{\tilde{m}_u}^{\alpha,\sigma,p} u(t_{k+1}) + \tilde{R}_{k+1}^u,
\end{align}
with $\tilde{R}_{k+1}^u = \mathcal{O} (h^{\sigma_{\tilde{m}_u+1}+\alpha+1} t_{k + 1}^{-\alpha-1}) + \mathcal{O} (h^{\alpha+p+1} t_{k+1}^{\sigma_{\tilde{m}_u+1}-\alpha-p-1}) + \mathcal{O}(h^{p+1})$. The history load correction weights $\left\{\tilde{W}_{k,j}^{(\alpha,\sigma,p)}\right\}$ are chosen such that \eqref{2.21} is exact for $u(t) = t^{\sigma_r}~(1 \le r \le m_{\tilde{u}})$. However, we remark that it is very difficult to obtain the analytical solution of $\mathcal{H}^k (t_{k+1})$, given $u(t) = t^{\sigma_r}$. Fortunately, we know from \cite{Deng2007} that \eqref{2.9} is also equivalent to
\begin{align}\label{2.22}
u(t_{k+1}) =& u(t_k) + \frac{\lambda}{\Gamma(\alpha)} \int_{t_k}^{t_{k+1}} (t_{k+1} - v)^{\alpha - 1} u(v) dv \nonumber\\
& + \frac{\lambda}{\Gamma(\alpha)} \int_{0}^{t_{k}} \left[(t_{k+1} - v)^{\alpha - 1} - (t_{k} - v)^{\alpha - 1}\right] u(v) dv.
\end{align}
Comparing \eqref{2.9} with \eqref{2.22} and using \eqref{2.5}, we can obtain
\begin{align}\label{2.23}
\mathcal{H}^k (t_{k+1}) =& - \frac{\lambda}{\Gamma(\alpha)} \int_{0}^{t_{k}} \left[(t_{k+1} - v)^{\alpha - 1} - (t_{k} - v)^{\alpha - 1}\right] u(v) dv \nonumber\\
=& - \frac{1}{\Gamma(\alpha)} \int_{0}^{t_{k}} \left[(t_{k+1} - v)^{\alpha - 1} - (t_{k} - v)^{\alpha - 1}\right] {}_0^C D_{t}^\alpha u(v) dv.
\end{align}
Therefore, we have the following linear system of size $\tilde{m}_u \times \tilde{m}_u$ to be solved for $k=1,2,\ldots,N-1$ time steps:
\begin{align}\label{2.24}
\sum_{j=1}^{\tilde{m}_u} \tilde{W}_{k,j}^{(\alpha,\sigma,p)} j^{\sigma_r} = & k^{\sigma_r} - \frac{\Gamma(\sigma_r+1)}{\Gamma(\alpha) \Gamma(\sigma_r-\alpha+1)} (k+1)^{\sigma_r} B\left(\frac{k}{k+1};\sigma_r-\alpha+1,\alpha\right) \nonumber\\
& - \frac{1}{\Gamma(\alpha) \Gamma(2-\alpha)} \sum_{j=0}^k \gamma_{k,j}^{(p)} j^{\sigma_r},~~~1 \le k \le N-1,~~~1 \le r \le \tilde{m}_u,
\end{align}
where $B(z;a,b)$ denotes the incomplete beta function, which is defined by
\begin{align*}
B(z;a,b) = \int_0^z v^{a-1} (1-v)^{b-1} dv.
\end{align*}

Substituting \eqref{2.18} and \eqref{2.21} into \eqref{2.9} yields
\begin{align}\label{2.25}
u(t_{k+1}) = u(t_k) \!+\! \lambda h^\alpha I_{h,m_u}^{\alpha,\sigma,p} u(t_{k+1}) \!-\! \mathcal{H}_{\tilde{m}_u}^{\alpha,\sigma,p} u(t_{k+1}) \!+\! R_{k+1}^u \!+\! \tilde{R}_{k+1}^u,~~~p=0, 1.
\end{align}
Dropping the truncation errors $R_{k+1}^u$ and $\tilde{R}_{k+1}^u$ in \eqref{2.25} and replacing $u(t_k)$ with approximate solution $u_k$, we obtain the following FAMMs with correction terms for solving \eqref{2.5}:
\begin{equation}\label{2.26}
\dfrac{u_{k+1} - u_k}{h^\alpha} = \lambda I_{h,m_u}^{\alpha,\sigma,p} u_{k+1} - \dfrac{1}{h^\alpha} \mathcal{H}_{\tilde{m}_u}^{\alpha,\sigma,p} u_{k+1},~~~p=0, 1,
\end{equation}
with $I_{h,m_u}^{\alpha,\sigma,p} u_{k+1}$ and $\mathcal{H}_{\tilde{m}_u}^{\alpha,\sigma,p} u_{k+1}$ given, respectively, by:
\begin{align}\label{2.27}
I_{h,m_u}^{\alpha,\sigma,p} u_{k+1} = \sum_{j=0}^p \beta_j^{(p)} u_{k+1-j} + \sum_{j=1}^{m_u} W_{k,j}^{(\alpha,\sigma,p)} \left(u_{j} - u_0\right),
\end{align}
and
\begin{align}\label{2.28}
\mathcal{H}_{\tilde{m}_u}^{\alpha,\sigma,p} u_{k+1} = \frac{1}{\Gamma(\alpha) \Gamma(2-\alpha)} \sum_{j=0}^{k} \gamma_{k,j}^{(p)} u_j + \sum_{j=1}^{\tilde{m}_u} \tilde{W}_{k,j}^{(\alpha,\sigma,p)} \left(u_{j} - u_0\right).
\end{align}

\subsection{Nonlinear FDEs}
\label{sec:nonlinear}

Having defined the discretization and corrections for the linear case, we now consider the numerical solutions of the following nonlinear FDE:
\begin{align}\label{3.1}
{}^C_{0} D_t^\alpha u(t) = \lambda u(t) + f(t,u),~~\alpha \in (0,1],~t \in (0,T];~~~u(0)=u_0,
\end{align}
where the nonlinear function $f: (0,T] \times \mathbb{R}^d \to \mathbb{R}^d$ satisfies the Lipschitz condition with constant $L > 0$:
\begin{align}\label{3.2}
\left\| f(t,u) - f(t,\hat{u}) \right\|_\infty \le L \left\| u - \hat{u} \right\|_\infty,~~~t \in (0,T],~u,\hat{u} \in \mathbb{R}^d,
\end{align}
where $\|\cdot\|_\infty$ denotes the usual maximum norm on $\mathbb{R}^d$. Under these assumptions, it has been proved by Diethelm and Ford \cite[Theorem 2.1 and 2.2]{Diethelm2002} that problem \eqref{3.1} has a unique solution on the interval $(0,T]$. 

By the same token, we adopt the FAMMs developed in \cite{Zayernouri2016MS} to solve \eqref{3.1}, and in a similar fashion as Section \ref{sec:linear}, but with the addition of the nonlinear term $f(t, u)$, we have:
\begin{align}\label{3.3}
u(t_{k+1})-u(t_k) = \left[ \lambda\, {}_{t_k}I_t^\alpha u(t) + {}_{t_k}I_t^\alpha f(t,u) -  \mathcal{H}^k (t)\right]\Big|_{t=t_{k+1}}.
\end{align}
Therefore, the FAMMs for \eqref{3.3} is given by:
\begin{align}\label{3.4}
\frac{u_{k+1} - u_k}{h^\alpha} = \sum_{j=0}^p \beta_j^{(p)} \left(\lambda u_{k+1-j} + f_{k+1-j} \right) - \frac{1}{h^\alpha}\mathcal{H}_{p}^k,~~~p=0, 1,
\end{align}
where $f_{k+1-j} = f(t_{k+1-j},u_{k+1-j})$. 

\subsubsection{Corrections Terms for $f(t,u)$}

The regularity of $f(t,u)$ is related to the regularity of $u(t)$. If $u(t)$ satisfies \eqref{2.17}, we know from \eqref{3.1} that
\begin{align}\label{3.5}
f(t,u) - f(0,u(0)) = & -\lambda \sum_{r=1}^{m+1} d_r t^{\sigma_r} + \sum_{r=1}^{m+1} d_r \frac{\Gamma(\sigma_r+1)}{\Gamma(\sigma_r-\alpha+1)} t^{\sigma_r - \alpha} + \ldots \nonumber\\
:=& \sum_{r=1}^{l+1} h_r t^{\delta_r} + \hat{\zeta}(t) t^{\delta_{l+2}},~~~\delta_r < \delta_{r+1},
\end{align}
where $\hat{\zeta}(t)$ is a uniformly continuous function for $t \in [0,T]$ and $\delta_r \in \left\{ \sigma_l\right\} \cup \left\{ \sigma_l - \alpha\right\}$. Similar to \eqref{2.18} and by using Lemma \ref{lemma2.1}, we have
\begin{align}\label{3.6}
{}_{t_k}I_t^\alpha f(t,u) \Big|_{t=t_{k+1}} 
=& h^\alpha \sum_{j=0}^p \beta_j^{(p)} f(t_{k+1-j},u(t_{k+1-j})) \nonumber\\
&+ h^\alpha \sum_{j=1}^{m_f} W_{k,j}^{(\alpha,\delta,p)} \left(f(t_j,u(t_{j})) - f(0,u_0)\right) + R_{k+1}^f \nonumber\\
:= & h^\alpha I_{h,m_f}^{\alpha,\delta,p} f(t_{k+1},u(t_{k+1})) + R_{k+1}^f,
\end{align}
where $R_{k+1}^f = \mathcal{O}(h^{\alpha+p+1} t_{k+1}^{\delta_{m_f+1}-p-1})$, and $\left\{W_{k,j}^{(\alpha,\delta,p)}\right\}$ with $k = 0,1,\ldots,N-1$ and $j=1,2,\ldots,m_f$ are given by:
\begin{align}\label{eq:corr_force1}
\sum_{j=1}^{m_f} W_{0,j}^{(\alpha,\delta,p)} j^{\delta_r} = \frac{\Gamma(\delta_r+1)}{\Gamma(\delta_r+\alpha+1)} - \sum_{j=0}^p \beta_j^{(p)} (1-j)^{\delta_r},~~~1 \le r \le m_f,
\end{align}
\begin{align}\label{eq:corr_force2}
\sum_{j=1}^{m_f} W_{k,j}^{(\alpha,\delta,p)} j^{\delta_r} =& \frac{k^{\delta_r}}{\Gamma(\alpha+1)} {}_2F_1\left(-\delta_r,1;\alpha+1;-\frac{1}{k}\right) - \sum_{j=0}^p \beta_j^{(p)} (k+1-j)^{\delta_r}, \nonumber\\
&1 \le k \le N-1,~~~1 \le r \le m_f.
\end{align}

Inserting \eqref{2.18}, \eqref{2.21} and \eqref{3.6} into \eqref{3.3} yields
\begin{align}\label{3.7}
u(t_{k+1}) = & u(t_k) + \lambda h^\alpha I_{h,m_u}^{\alpha,\sigma,p} u(t_{k+1}) + h^\alpha I_{h,m_f}^{\alpha,\delta,p} f(t_{k+1},u(t_{k+1})) \nonumber\\
& - \mathcal{H}_{\tilde{m}_u}^{\alpha,\sigma,p} u(t_{k+1}) + R_{k+1}^u + R_{k+1}^f + \tilde{R}_{k+1}^u,~~~p=0, 1.
\end{align}
Therefore, we obtain the following FAMMs with correction terms for solving \eqref{3.1}:
\begin{align}\label{3.8}
\dfrac{u_{k+1} - u_k}{h^\alpha} = \lambda I_{h,m_u}^{\alpha,\sigma,p} u_{k+1} + 
I_{h,m_f}^{\alpha,\delta,p} f_{k+1} - 
\dfrac{1}{h^\alpha}\mathcal{H}_{\tilde{m}_u}^{\alpha,\sigma,p} u_{k+1},~~~p=0, 
1,
\end{align}
with $I_{h,m_u}^{\alpha,\sigma,p} u_{k+1}$, $\mathcal{H}_{\tilde{m}_u}^{\alpha,\sigma,p} u_{k+1}$ and $I_{h,m_f}^{\alpha,\delta,p} f_{k+1}$ given, respectively, by \eqref{2.27}, \eqref{2.28} and
\begin{align}\label{3.8*}
I_{h,m_f}^{\alpha,\delta,p} f_{k+1} = \sum_{j=0}^p \beta_j^{(p)} f_{k+1-j} + \sum_{j=1}^{m_f} W_{k,j}^{(\alpha,\delta,p)} \left(f_j - f_0\right).
\end{align}

In order to obtain the IMEX methods, we follow the idea from Cao et al. \cite{Cao2016}, and employ an extrapolation to linearize the nonlinear force term $f(t_{k+1},u(t_{k+1}))$ in \eqref{3.7}, which is given by:
\begin{align}\label{3.9}
f(t_{k+1},u(t_{k+1})) = E_k^{p} f(t_k,u(t_k)) \!+\! \sum_{j=1}^{\tilde{m}_f} W_{k,j}^{(\delta,p)} (f(t_j,u(t_j)) - f(0,u_0)) + \tilde{R}_{k+1}^f,
\end{align}
where 
\begin{align}\label{3.10}
E_k^p f(t_k,u(t_k))= \left\{ {\begin{array}{*{20}{l}}
f(t_{k},u(t_k)),&p=0,\\[3\jot]
2f(t_{k},u(t_k)) - f(t_{k-1},u(t_{k-1})),~~~&p=1,
\end{array}} \right.
\end{align}
and $\tilde{R}_k^f = \mathcal{O}(h^{p+1} t_{k+1}^{\delta_{\tilde{m}_f + 1}-p-1})$. In addition, the $\tilde{m}_f \times \tilde{m}_f$ linear system for correction weights is given by:
\begin{align}\label{3.11}
\sum_{j=1}^{\tilde{m}_f} W_{k,j}^{(\delta,0)} j^{\delta_r} = (k+1)^{\delta_r} - k^{\delta_r},~~~1 \le r \le \tilde{m}_f,
\end{align}
\begin{align}\label{3.12}
\sum_{j=1}^{\tilde{m}_f} W_{k,j}^{(\delta,1)} j^{\delta_r} = (k+1)^{\delta_r} - 2k^{\delta_r} + (k-1)^{\delta_r},~~~1 \le r \le \tilde{m}_f.
\end{align}

Inserting \eqref{3.9} into \eqref{3.7} yields
\begin{align}\label{3.13}
u(t_{k+1}) = & u(t_k) + \lambda h^\alpha I_{h,m_u}^{\alpha,\sigma,p} u(t_{k+1}) + h^\alpha I_{h,m_f}^{\alpha,\delta,p} f(t_{k+1},u(t_{k+1})) \nonumber\\
& + h^\alpha \beta_0^{(p)} \bigg[-f(t_{k+1},u(t_{k+1})) + E_k^p f(t_{k},u(t_k)) \nonumber\\
&+ \sum_{j=1}^{\tilde{m}_f} W_{k,j}^{(\delta,p)} (f(t_j,u(t_j)) - f(0,u_0)) \bigg] - \mathcal{H}_{\tilde{m}_u}^{\alpha,\sigma,p} u(t_{k+1}) \nonumber\\
& + R_{k+1}^u + R_{k+1}^f + \tilde{R}_{k+1}^u + h^\alpha \beta_0^{(p)} \tilde{R}_{k+1}^f,~~~p=0, 1.
\end{align}
Finally, we obtain the IMEX($p$) methods in the following form:
\begin{align}\label{3.14}
\dfrac{u_{k+1} - u_k}{h^\alpha} =& \lambda I_{h,m_u}^{\alpha,\sigma,p} u_{k+1} + I_{h,m_f}^{\alpha,\delta,p} f_{k+1} - \dfrac{1}{h^\alpha}\mathcal{H}_{\tilde{m}_u}^{\alpha,\sigma,p} u_{k+1} \nonumber\\
& + \beta_0^{(p)} \left[-f_{k+1} + E_k^p f_{k} + \sum_{j=1}^{\tilde{m}_f} W_{k,j}^{(\delta,p)} (f_j - f_0) \right],~~~p=0, 1,
\end{align}
with $I_{h,m_u}^{\alpha,\sigma,p} u_{k+1}$, $\mathcal{H}_{\tilde{m}_u}^{\alpha,\sigma,p} u_{k+1}$, $I_{h,m_f}^{\alpha,\delta,p} f_{k+1}$ and $E_k^p f_{k}$ given, respectively, by \eqref{2.27}, \eqref{2.28}, \eqref{3.8*} and
\begin{align}\label{3.10*}
E_k^p f_k= \left\{ {\begin{array}{*{20}{l}}
f_{k},&p=0,\\[3\jot]
2f_{k} - f_{k-1},~~~&p=1.
\end{array}} \right.
\end{align}

\begin{remark}
If $\alpha = 1$, the history load term for \eqref{3.1} is zero and we don't need to add the correction terms to solve \eqref{3.1}. Moreover, \eqref{3.14} recovers the standard IMEX methods.
\end{remark}

In the following, we will present the convergence results for the IMEX($p$) methods \eqref{3.14}. For this purpose, we first introduce some preparatory results. Both of the proofs will be given in Appendix \ref{appendix}.

\begin{lemma}\label{lemma2.4}
The correction weights $W_{k,j}^{(\alpha,\sigma,p)}$, $\tilde{W}_{k,j}^{(\alpha,\sigma,p)}$, $W_{k,j}^{(\alpha,\delta,p)}$ and $W_{k,j}^{(\delta,p)}$ in \eqref{2.18}, \eqref{2.21}, \eqref{3.6} and \eqref{3.9}, respectively, satisfy
\begin{align*}
\left|W_{k,j}^{(\alpha,\sigma,p)} \right| = \mathcal{O}((k+1)^{\sigma_{m_u}-p-1}),~~~\left|W_{k,j}^{(\alpha,\delta,p)}\right| = \mathcal{O}((k+1)^{\delta_{m_f}-p-1}),
\end{align*}
\[\left|\tilde{W}_{k,j}^{(\alpha,\sigma,p)}\right| = \mathcal{O} ((k+1)^{-\alpha-1}) + \mathcal{O}((k+1)^{\sigma_{\tilde{m}_u}-p-1}),~~~\left| W_{k,j}^{(\delta,p)} \right| = \mathcal{O}((k+1)^{\delta_{\tilde{m}_f} - p - 1}).\]
\end{lemma}

\begin{theorem}\label{theorem:convergence}
Suppose that the Lipschitz condition \eqref{3.2} holds. If $\sigma_{\tilde{m}_u} \le p+1,~\sigma_{m_u},\delta_{m_f},\delta_{\tilde{m}_f} \le p+\alpha+1$, then, for the IMEX($p$) methods \eqref{3.14}, there exists a constant $C_4 > 0$ independent of $h$ such that
\begin{align}\label{3.15}
\max_{0 \le k \le N} \left\| u(t_k) - u_k\right\|_{\infty} \le C_4 \left(\sum_{j=1}^M \left\| u(t_j) - u_j\right\|_{\infty} + h^q \right),
\end{align}
where $M = \max\left\{m_u, m_f, \tilde{m}_u, \tilde{m}_f \right\}$ and 
\[q = \min\left\{p+1, \sigma_{\tilde{m}_u+1} + p, \sigma_{m_u+1}+\alpha+p, \delta_{m_f+1}+\alpha+p, \delta_{\tilde{m}_f+1}+\alpha+p \right\}.\]
\end{theorem}

\section{Linear Stability of IMEX($p$) Methods}
\label{section3}

In this section, we investigate the linear stability of the proposed IMEX($p$) methods \eqref{3.14} by considering the following usual scalar test equation
\begin{align}\label{4.1}
{}^C_{0} D_t^\alpha u(t) = \lambda u(t) + \rho u(t),~~\alpha \in (0,1],~t \ge 0,~\lambda, \rho \in \mathbb{C};~~~u(0)=u_0.
\end{align}
For this, the following result from \cite{Lubich19861} is useful to determine the stability regions of the obtained numerical schemes.

\begin{lemma}(cf. \cite{Lubich19861})\label{lemma4.1}
Assume that the sequence $\left\{g_k \right\}$ is convergent and that the quadrature weights $w_k (k \ge 1)$ satisfy
\begin{align*}
w_k = \frac{k^{\alpha - 1}}{\Gamma(\alpha+1)} + v_k,~\mbox{where}~\sum_{k=1}^\infty |v_k| < \infty,
\end{align*}
then the stability region of the convolution quadrature $u_k = g_k + \hat{h}\sum\limits_{j=0}^k w_{k-j} u_j$ is
\[S = \left\{ \hat{h} \in \mathbb{C}: 1 - \hat{h}w^\alpha (\xi) \ne 0,~|\xi| \le 1\right\},~\mbox{where}~w^\alpha(\xi) = \sum_{j=0}^\infty w_j \xi^j,\]
where $\hat{h} = h^\alpha (\lambda+\rho)$ or $\hat{h}$ is some function of $h^\alpha (\lambda+\rho)$.
\end{lemma}

We first consider the linear stability of the IMEX(0) for the test equation \eqref{4.1}, it holds that
\begin{align}\label{4.2}
u_{k+1} =& u_k + \frac{\lambda h^\alpha}{\Gamma(\alpha + 1)} u_{k+1} +  h^\alpha\left( \lambda\sum_{j=1}^{m_u} W_{k,j}^{(\alpha,\sigma,0)} + \rho \sum_{j=1}^{m_f} W_{k,j}^{(\alpha,\delta,0)} \right) (u_{j} - u_0) \nonumber\\
& + \frac{\rho h^\alpha}{\Gamma(\alpha+1)} \left[u_{k} + \sum_{j=1}^{\tilde{m}_f} W_{k,j}^{(\delta,0)} (u_j - u_0) \right] - \frac{1}{\Gamma(\alpha) \Gamma(2-\alpha)} \sum_{j=0}^{k} \gamma_{k,j}^{(0)} u_j \nonumber\\
& - \sum_{j=1}^{\tilde{m}_u} \tilde{W}_{k,j}^{(\alpha,\sigma,0)} (u_{j} - u_0) \nonumber\\
= & u_k \!+\! \frac{h^\alpha}{\Gamma(\alpha \!+\! 1)} (\lambda u_{k+1} \!+\! \rho u_{k}) \!-\! \frac{1}{\Gamma(\alpha) \Gamma(2\!-\!\alpha)} \sum_{j=0}^{k} \gamma_{k,j}^{(0)} u_j \!+\! \sum_{j=1}^M W_{k,j} (u_{j} \!-\! u_0),
\end{align}
where
\begin{align*}
\sum_{j=1}^M W_{k,j} (u_{j} - u_0) = &  \lambda h^\alpha \sum_{j=1}^{m_u} W_{k,j}^{(\alpha,\sigma,0)} (u_{j} - u_0) + \rho h^\alpha \sum_{j=1}^{m_f} W_{k,j}^{(\alpha,\delta,0)} (u_{j} - u_0) \nonumber\\
& + \frac{\rho h^\alpha}{\Gamma(\alpha+1)} \sum_{j=1}^{\tilde{m}_f} W_{k,j}^{(\delta,0)} (u_j - u_0) - \sum_{j=1}^{\tilde{m}_u} \tilde{W}_{k,j}^{(\alpha,\sigma,0)} (u_{j} - u_0).
\end{align*}
Since $\sum\limits_{j=1}^M W_{k,j} (u_{j} - u_0)$ does not affect the stability analysis, so we don't give the exact expression of $W_{k,j}$. Denote $U(\xi) = \sum\limits_{k=0}^\infty u_k \xi^k,~|\xi| \le 1$. Then it follows from \eqref{4.2} that
\begin{align*}
\sum_{k=0}^\infty u_{k+1} \xi^k =& \sum_{k=0}^\infty u_k \xi^k + \frac{\lambda h^\alpha}{\Gamma(\alpha + 1)} \sum_{k=0}^\infty u_{k+1} \xi^k + \frac{\rho h^\alpha}{\Gamma(\alpha + 1)} \sum_{k=0}^\infty u_{k} \xi^k \nonumber\\
&- \frac{1}{\Gamma(\alpha)\Gamma(2-\alpha)} \sum_{k=0}^\infty \left(\sum_{j=0}^{k} \gamma_{k,j}^{(0)} u_{j} \right) \xi^k + \sum_{k=0}^\infty \sum_{j=1}^M W_{k,j} (u_j - u_0) \xi^k,
\end{align*}
which leads to
\begin{align}\label{4.3}
\frac{1}{\xi} (U(\xi) &- u_0) = U(\xi) + \frac{\lambda h^\alpha}{\Gamma(\alpha + 1) \xi} \left(U(\xi) - u_0\right) + \frac{\rho h^\alpha}{\Gamma(\alpha + 1)}U(\xi) \nonumber\\
& - \frac{1}{\Gamma(\alpha)\Gamma(2-\alpha)} \gamma_{k}^{(0)}(\xi) U(\xi) + \sum_{k=1}^\infty \sum_{j=1}^M W_{k,j} (u_j - u_0) \xi^k,
\end{align}
where $\gamma_{k}^{(0)}(\xi) = \sum\limits_{j=0}^\infty \gamma^{(0)}_{k,j} \xi^j$. We simplify \eqref{4.3} as
\begin{align*}
&\left[1-\xi-\frac{\lambda h^\alpha}{\Gamma(\alpha + 1)} - \frac{\rho h^\alpha}{\Gamma(\alpha + 1)} \xi + \frac{1}{\Gamma(\alpha)\Gamma(2-\alpha)} \gamma_{k}^{(0)}(\xi) \xi \right] U(\xi) = H(\xi),
\end{align*}
where
\begin{align*}
H(\xi) = &\sum_{k=0}^\infty H_k \xi^k = \left[1- \frac{\lambda h^\alpha}{\Gamma(\alpha + 1)} \right] u_0 + \sum_{k=1}^\infty \sum_{j=1}^m W_{k,j} (u_j - u_0) \xi^{k+1}.
\end{align*}
By using Lemma \ref{lemma2.4}, when 
\[\sigma_{m_u},\delta_{m_f},\delta_{\tilde{m}_f} < 1,~~~\sigma_{\tilde{m}_u} < \alpha +1,\] 
we can obtain that $\left\{H_k \right\}$ is a convergent sequence. Moreover, we know from Lemma \ref{lemma2.3} that $\sum\limits_{j=0}^\infty \left|\gamma_{k,j}^{(0)}\right| < \infty$. Then it follows from Lemma \ref{lemma4.1} that method IMEX(0) is stable if 
\begin{align*}
1-\xi-\frac{\lambda h^\alpha}{\Gamma(\alpha + 1)} - \frac{\rho h^\alpha}{\Gamma(\alpha + 1)} \xi + \frac{1}{\Gamma(\alpha)\Gamma(2-\alpha)} \gamma_{k}^{(0)}(\xi) \xi \ne 0,~~~\forall |\xi| \le 1.
\end{align*}
Similarly, we can obtain the stability region of the method IMEX(1). Then we have the following theorem.

\begin{theorem}\label{theorem:stability}
Let $\rho = \kappa \lambda$ and $\hat{h} = \lambda h^\alpha$. Then for $\sigma_{m_u},\delta_{m_f},\delta_{\tilde{m}_f} < p+1,~\sigma_{\tilde{m}_u} < \alpha + p + 1$, we have the stability region of IMEX(0):
\begin{align*}
S_{0} = \mathbb{C} \setminus \left\{ \hat{h} : \hat{h} = \frac{\Gamma(\alpha+1)}{1 + \kappa\xi} \left(1 - \xi + \frac{\gamma_{k}^{(0)}(\xi) \xi}{\Gamma(\alpha)\Gamma(2-\alpha)} \right),~|\xi| \le 1\right\}.
\end{align*}
and the stability region of IMEX(1):
\begin{align*}
S_{1} = \mathbb{C} \setminus \left\{ \hat{h} : \hat{h} = \frac{\Gamma(\alpha+2)}{(1\!+\!\kappa)(\alpha\xi \!+\! 1) \!-\! \kappa (\xi\!-\!1)^2} \left(1 - \xi + \frac{\gamma_{k}^{(1)}(\xi) \xi}{\Gamma(\alpha)\Gamma(2-\alpha)} \right),~|\xi| \le 1\right\},
\end{align*}
where $\gamma_{k}^{(1)}(\xi) = \sum\limits_{j=0}^\infty \gamma^{(1)}_{k,j} \xi^j$.
\end{theorem}

In Figure \ref{fig:1} (a)-(c), we plot the stability regions of the method IMEX(0) with $\alpha = 0.2,0.5,0.8$ and $\rho = 0.5 \lambda$, respectively. We also plot the stability regions of the method IMEX(1) with $\alpha = 0.2,0.5,0.8$ and $\rho = 0.5 \lambda$ in Figure \ref{fig:1} (d). As the functions $\gamma_{k}^{(p)}(\xi) = \sum\limits_{j=0}^\infty \gamma_{k,j}^{(p)} \xi^j$ is not explicitly known, so in all these figures we take $k=10^5$.

\begin{figure}[!htbp]
\subfigure[]{ \centering
\includegraphics[width=0.48\textwidth,height=0.38\textwidth]{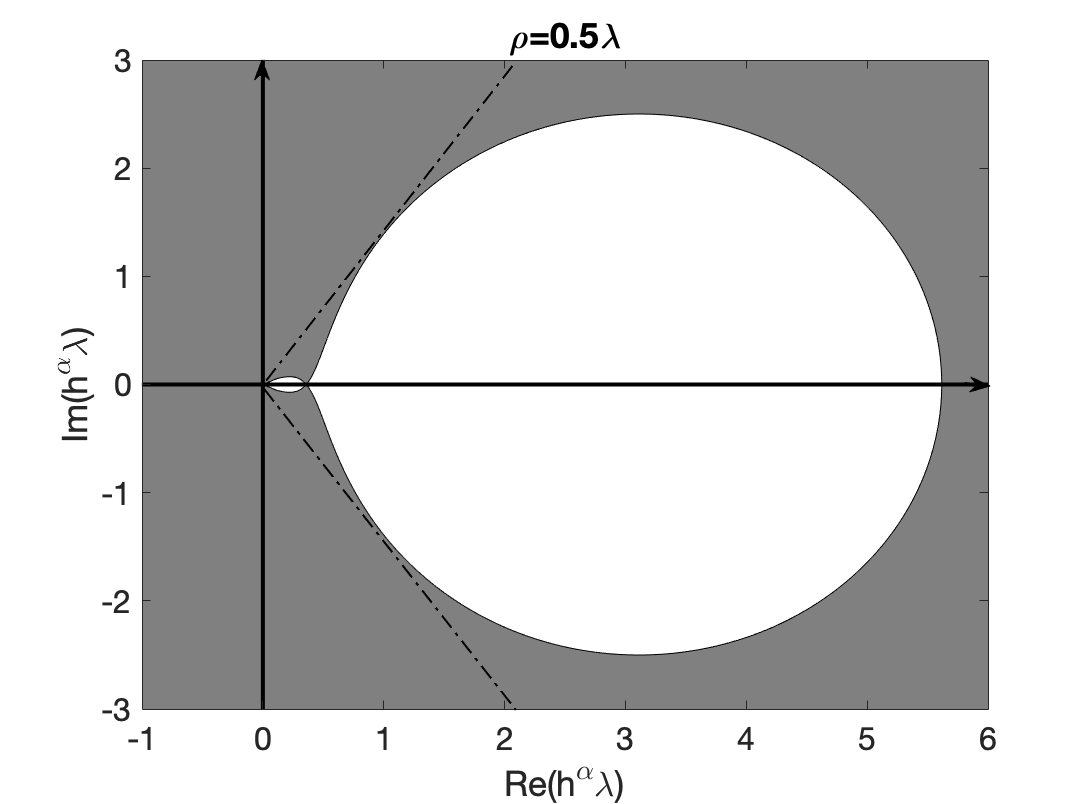}}
\subfigure[]{ \centering
\includegraphics[width=0.48\textwidth,height=0.38\textwidth]{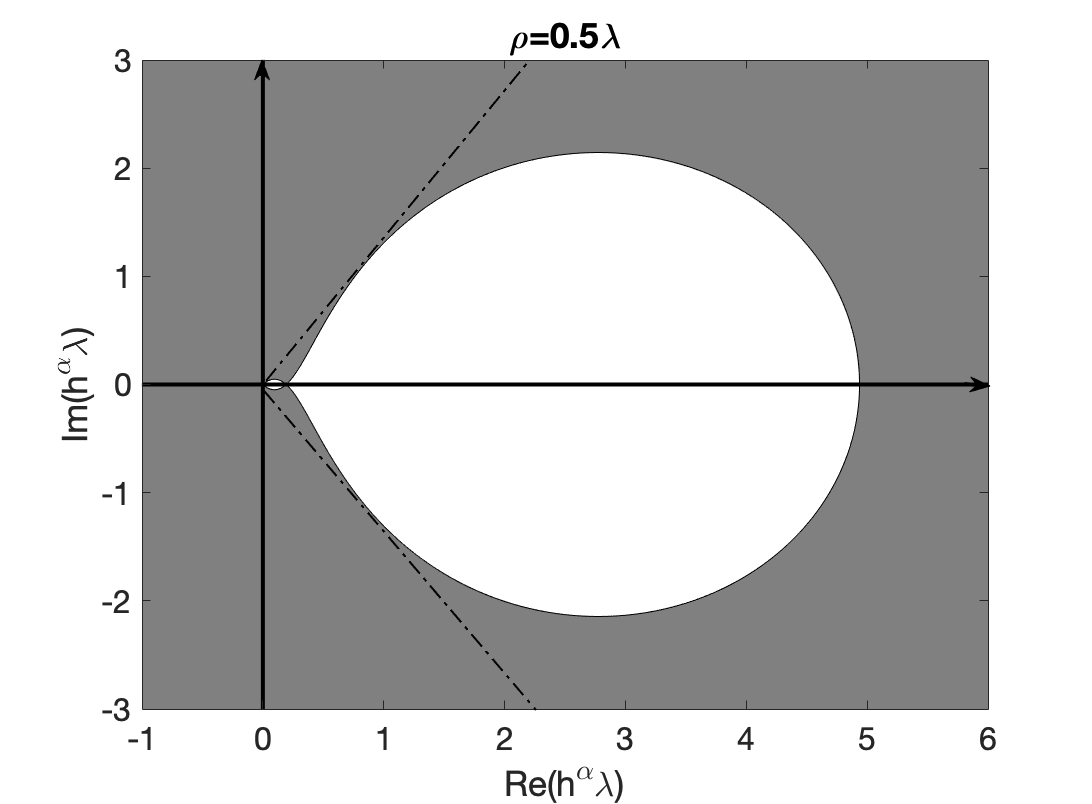}}
\subfigure[]{ \centering
\includegraphics[width=0.48\textwidth,height=0.38\textwidth]{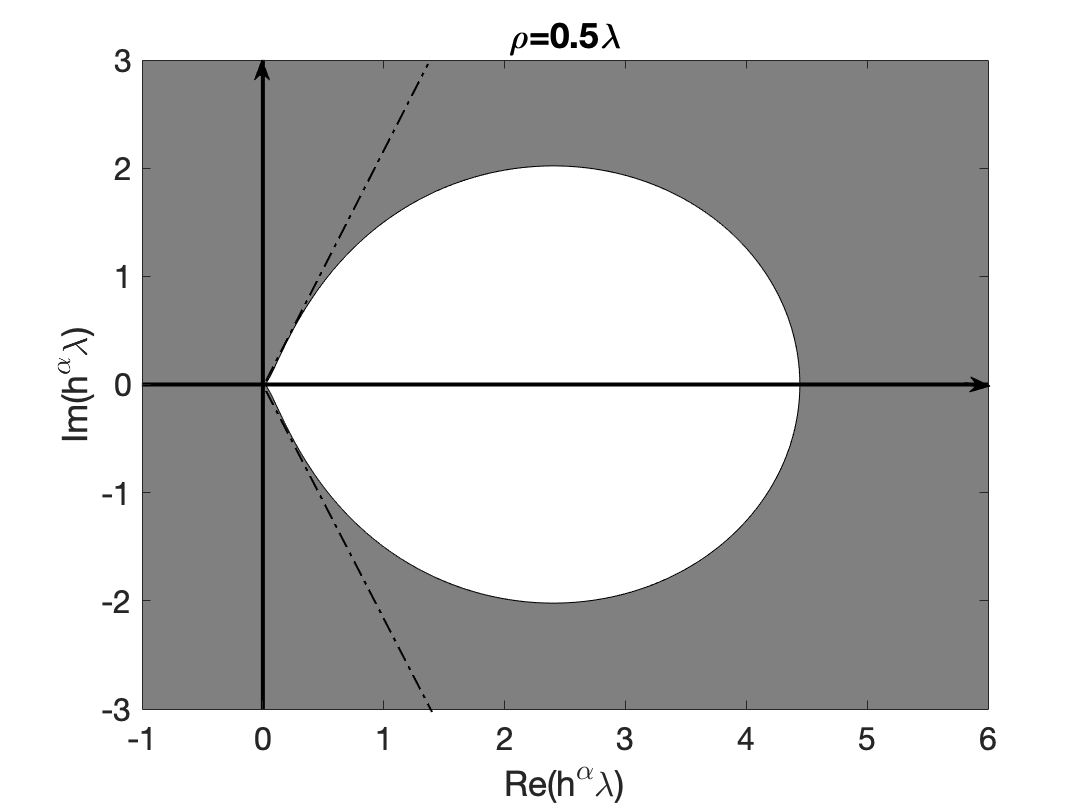}}
\subfigure[]{ \centering
\includegraphics[width=0.48\textwidth,height=0.38\textwidth]{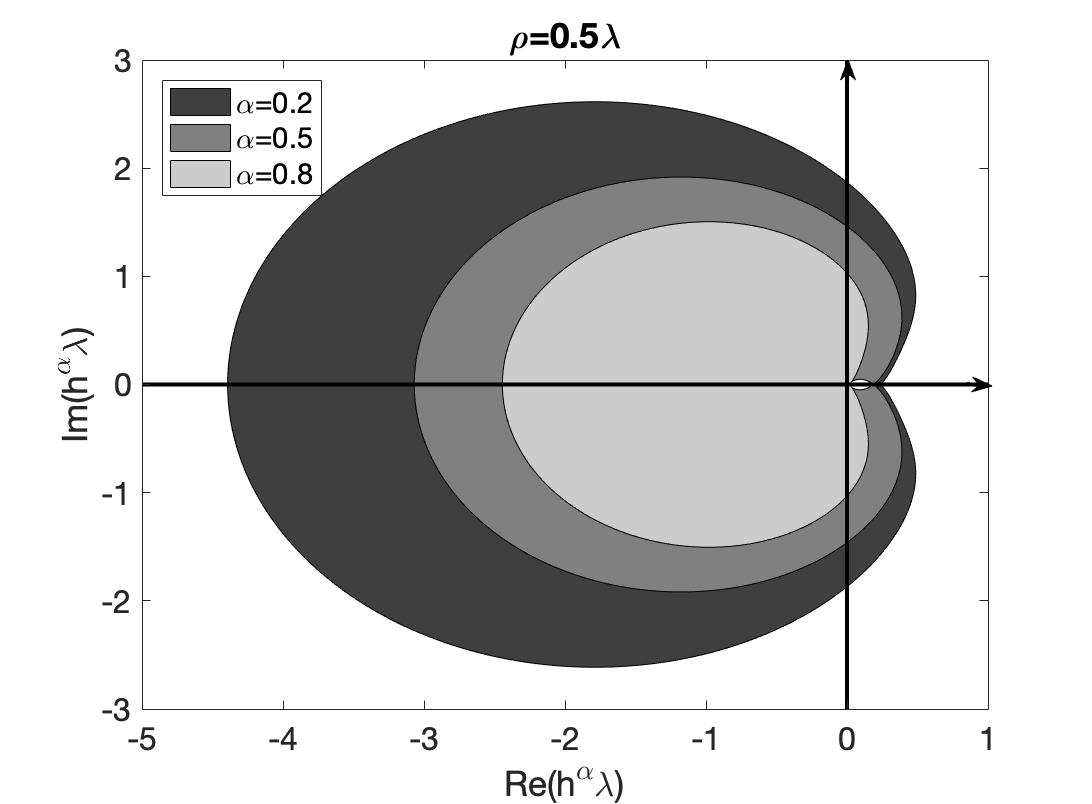}}
\caption{(a) Stability region of IMEX(0) with $\alpha = 0.2$ and $\rho = 0.5 \lambda$; (b) Stability region of IMEX(0) with $\alpha = 0.5$ and $\rho = 0.5 \lambda$; (c) Stability region of IMEX(0) with $\alpha = 0.8$ and $\rho = 0.5 \lambda$; (d) Stability region of the IMEX(1) with different $\alpha$ and $\rho = 0.5 \lambda$.}
\label{fig:1}
\end{figure}

\section{The Fast Implementation of IMEX($p$) Methods}
\label{section4}

The step-by-step numerical solution of \eqref{3.14} for $N$ time-steps requires $\mathcal{O}(N^2)$ evaluations of the time-dependent coefficients given by the hypergeometric functions, making the scheme expensive. Hence, we rewrite \eqref{3.14} as the matrices form, where the corresponding convolution matrices of coefficients have the Toeplitz structure and thus we leverage the use of FFTs to obtain the solution of the problem with complexity $\mathcal{O}(N \log N)$. For simplicity and objectivity, we demonstrate the procedure only for the IMEX(0) scheme, for which we start by introducing the notations 
\[U=(u_{1},u_{2},\ldots,u_N)^T,~~~F(U)=(f_{1},f_{2},\ldots,f_N)^T,\]
where $U$ denotes the unknown solution vector. Then IMEX(0) can be written in a compact form:
\begin{align}\label{5.1}
(A\otimes I_d - \lambda B \otimes I_d)U = h^\alpha (B \otimes I_d + C \otimes I_d) F(U) + D,
\end{align}
where $\otimes$ denotes the Kronecker product and $I_d$ represents the $d \times d$ identity matrix. Here the value of $d$ can represents, for instance, the number of equations for a system of nonlinear FDEs. Furthermore, we have:
\begin{align}\label{eq:A}
A= \left[\! {\begin{array}{@{}ccccc@{}}
{1}&{}&{}&{}&{}\\[2\jot]
{-1 \!+\! \dfrac{\gamma_{1,1}^{(0)}}{\Gamma(\alpha) \Gamma(2\!-\!\alpha)}} & {1} & {} & {} & {}\\[2\jot]
{\dfrac{\gamma_{2,1}^{(0)}}{\Gamma(\alpha) \Gamma(2\!-\!\alpha)}} & {-1 \!+\! \dfrac{\gamma_{2,2}^{(0)}}{\Gamma(\alpha) \Gamma(2\!-\!\alpha)}} & {1} & {} & {}\\[2\jot]
\vdots & \vdots & \ddots &\ddots &{}\\[2\jot]
{\dfrac{\gamma_{N-1,1}^{(0)}}{\Gamma(\alpha) \Gamma(2\!-\!\alpha)}} & {\dfrac{\gamma_{N-1,2}^{(0)}}{\Gamma(\alpha) \Gamma(2\!-\!\alpha)}} & {\cdots} & {-1 \!+\! \dfrac{\gamma_{N-1,N-1}^{(0)}}{\Gamma(\alpha) \Gamma(2\!-\!\alpha)}} & {1}
\end{array}} \right] \in \mathbb{R}^{N{\times}N},
\end{align}
\begin{align}\label{eq:BC}
B = \frac{1}{\Gamma(\alpha+1)}I_{N},
~~~C = \frac{1}{\Gamma(\alpha+1)}\left[ {\begin{array}{@{}ccccc@{}}
{-1} & {} & {} & {} & {}\\
{1} & {-1} & {} & {} & {}\\
{0} & {1} & {-1} & {} & {}\\
\vdots & \ddots & \ddots & \ddots & {}\\
{0} & {\cdots} & {0} & {1} & {-1}
\end{array}} \right] \in \mathbb{R}^{N{\times}N},
\end{align}
and $D$ is a vector related to the corrections and solutions for the initial steps, in the following way:
\begin{align}\label{eq:Dvector}
    D = & \lambda h^\alpha (W_{m_u} \otimes I_d) U_{m_u} + h^\alpha (W_{m_f} \otimes I_d) F(U_{m_f}) - (W_{\tilde{m}_u} \otimes I_d) U_{\tilde{m}_u} \nonumber\\
    & + \frac{h^\alpha}{\Gamma(\alpha+1)} (W_{\tilde{m}_f} \otimes I_d ) F(U_{\tilde{m}_f}) + (a_0\otimes I_d) u_0 + h^\alpha (b_0\otimes I_d) f_0,
\end{align}
where the correction weights $W_{\hat{M}}~(\hat{M} = m_u, m_f, \tilde{m}_u, \tilde{m}_f)$ are denoted by the $N\times \hat{M}$ matrices with the element $W_{k,j}^{(\alpha,\sigma,0)}, W_{k,j}^{(\alpha,\delta,0)}, \tilde{W}_{k,j}^{(\alpha,\sigma,0)}$ and $\tilde{W}_{k,j}^{(\delta,0)}$, respectively, for $k=0,1,\ldots,N-1$ and $j=1, 2,\ldots,\hat{M}$. We also have,
\begin{equation*}
    a_0 = \left(1,\frac{\gamma_{1,0}^{(0)}}{\Gamma(\alpha) \Gamma(2-\alpha)},\frac{\gamma_{2,0}^{(0)}}{\Gamma(\alpha) \Gamma(2-\alpha)},\ldots,\frac{\gamma_{N-1,0}^{(0)}}{\Gamma(\alpha) \Gamma(2-\alpha)} \right)^T \in \mathbb{R}^N,
\end{equation*}
\begin{equation*}
    b_0 = \left(\frac{1}{\Gamma(\alpha+1)},0,\ldots,0\right)^T \in \mathbb{R}^N,~~~U_{\hat{M}} = \left( u_1 - u_0 ,u_2 - u_0, \ldots, u_{\hat{M}} - u_0\right)^T.
\end{equation*}
From \eqref{2.13}, we observe that $\mathcal{A}_{k+1,j+1} = \mathcal{A}_{k,j}$, and therefore $A$ is a lower-triangular Toeplitz matrix.

In what follows, we will analyze the unique solvability of the IMEX(0) scheme. For this purpose, we introduce the mapping $\Phi_h: \mathbb{R}^{Nd} \to \mathbb{R}^{Nd}$ as follows:
\[\Phi_h(X) := h^\alpha \left[(A - \lambda B)^{-1} (B + C) \otimes I_d \right] F(X) + \left[(A - \lambda B)^{-1} \otimes I_d \right] D,\]
and therefore, we have the following result.

\begin{theorem}
Suppose that Lipschitz condition \eqref{3.2} holds and
\begin{align}\label{5.2}
h^\alpha L \|(A - \lambda B)^{-1} (B + C)\|_\infty < 1. 
\end{align}
Then the method IMEX(0) has a unique solution $U \in \mathbb{R}^{Nd}$.
\end{theorem}

\begin{proof}
Let $X = \left(x_{1}^T,x_{2}^T,\ldots,x_N^T\right)^T,~\hat{X} = \left(\hat{x}_{1}^T,\hat{x}_{2}^T,\ldots,\hat{x}_N^T\right)^T$ be two arbitrary vectors in $\mathbb{R}^{Nd}$. It follows from the Lipschitz condition \eqref{3.2} that
\begin{align*}
\left\|\Phi_h(X) - \Phi_h(\hat{X}) \right\|_\infty \le & h^\alpha \left\|(A - \lambda B)^{-1} (B + C)\right\|_\infty \left\|F(X) - F(\hat{X})\right\|_\infty \nonumber\\
\le & h^\alpha L \left\|(A - \lambda B)^{-1} (B + C)\right\|_\infty \left\|X - \hat{X}\right\|_\infty.
\end{align*}
If condition \eqref{5.2} holds, we know that $\Phi_h(X)$ is a contraction mapping with contraction factor $h^\alpha L \|(A - \lambda B)^{-1} (B + C)\|_\infty$. Moreover, it is well known that space $\mathbb{R}^{Nd}$ with norm $\|\cdot\|_\infty$ is complete. Hence, according to the Banach contraction mapping principle (see e.g. \cite{Agarwal2001}), mapping $\Phi_h(X)$ has a unique fixed point in $\mathbb{R}^{Nd}$. Namely, the method IMEX(0) has a unique solution $U \in \mathbb{R}^{Nd}$.
~~~\end{proof}

\begin{remark}
It should be pointed that, the coefficient matrices $A$ for IMEX(1) is not the Toeplitz matrices, but we can choose the first column $a_1$ of $A$ as 
\[\hat{a}_1 = \left(1,-1 \!+\! \dfrac{\gamma_{2,2}^{(1)}}{\Gamma(\alpha) \Gamma(2\!-\!\alpha)},\dfrac{\gamma_{3,2}^{(1)}}{\Gamma(\alpha) \Gamma(2\!-\!\alpha)},\ldots,\dfrac{\gamma_{N,2}^{(1)}}{\Gamma(\alpha) \Gamma(2\!-\!\alpha)} \right)^T \in \mathbb{R}^N.\]
Then $A$ will be a Toeplitz matrix. Moreover, if we do this, the corresponding vector $D$ for IMEX(1) will be change by adding the term $\left[(\hat{a}_1 - a_1) \otimes I_d \right] U$.
\end{remark}

\subsection{Fast Approximate Inversion Scheme}
\label{Sec:FastInversion}

In order to obtain a fast solution to the Toeplitz system \eqref{5.1}, we employ the scheme developed in \cite{Lux2015}, which approximates the lower-triangular Toeplitz matrix $K^{-1}$. In particular, we have $K = (A\otimes I_d - \lambda B \otimes I_d)$ for the method IMEX(0) in \eqref{5.1}. The first step involves approximating the matrix $K$ by the following block $\epsilon$-circulant matrix:
\begin{equation}\label{eq:block_circulant}
    K_\epsilon = 
    \begin{bmatrix}
        K_0     & \epsilon K_{N-1} & \dots            & \epsilon K_2  & \epsilon K_1     \\
        K_1     & K_{0}            & \epsilon K_{N-1} & \dots         & \epsilon K_2     \\
        \vdots  & K_1              & K_0              & \ddots        & \vdots           \\
        K_{N-2} & \dots            & \ddots           & \ddots        & \epsilon K_{N-1} \\
        K_{N-1} & K_{N-2}          & \dots            & K_1           & K_0
    \end{bmatrix},
\end{equation}
with $\epsilon > 0$. It is reported by Lu \textit{et al.} \cite{Lux2015} that the accuracy of the fast inversion is $\mathcal{O}(\epsilon)$, where mathematically $\epsilon$ can be taken arbitrarily small. However, for double-precision arithmetic, $\epsilon$ cannot be set too small due to rounding errors. Numerical experiments demonstrated the smallest practical value to be $\epsilon = 5\times10^{-9}$. It is also shown in \cite{Lux2015} that $K^{-1}_\epsilon$ is also a block $\epsilon$-circulant matrix, and therefore the solution to system \eqref{5.1} can be written in the following way:
\begin{equation*}
    U  \approx K^{-1}_\epsilon R,
\end{equation*}
where $R$ denote the right-hand side of \eqref{5.1}. 

Let $D_{\varrho} = \mathrm{diag}(1,\varrho,\dots,\varrho^{N-1})$, with $\varrho = \epsilon^{1/N}$ be a diagonal matrix and $F_N$ be a $N\times N$ Fourier matrix. We then have the following spectral decomposition:
\begin{equation*}
    K^{-1}_\epsilon = \left[(D^{-1}_\varrho F^*_N) \otimes I_d\right] \mathrm{diag}\left(\Lambda^{-1}_0, \Lambda^{-1}_1, \dots, \Lambda^{-1}_{N-1} \right) \left[(F_N D_\varrho) \otimes I_d\right],
\end{equation*}
with 
\begin{equation}\label{5.3}
    \begin{bmatrix} \Lambda_0 \\ \Lambda_1 \\ \vdots \\ \Lambda_{N-1} \end{bmatrix} = 
    \left[ (\sqrt{N} F_{N} D_\varrho) \otimes I_d \right]
    \begin{bmatrix} K_0 \\ K_1 \\ \vdots \\ K_{N-1} \end{bmatrix}.
\end{equation}
Finally, the approximate solution for $U$ becomes:
\begin{equation}\label{5.4}
    U \approx \left[(D^{-1}_\varrho F^*_{N}) \otimes I_d\right] \mathrm{diag}\left(\Lambda^{-1}_0, \Lambda^{-1}_1, \dots, \Lambda^{-1}_{N-1} \right) \left[(F_{N} D_\varrho) \otimes I_d\right] R,
\end{equation}
where in practical implementations, we replace the Fourier matrices $F_{N}$ in \eqref{5.3} and \eqref{5.4} with FFT operations in order to achieve a computational complexity of $\mathcal{O}(N\log N)$. The other operations to form $R$ do not require FFTs, since matrices $B$ and $C$ are sparse, lower-Toeplitz nature, and the vector $D$ is formed through the multiplication of tall matrices with small vectors.

\begin{remark}
Note that \eqref{5.4} is a nonlinear system, therefore iterative solver should be applied to solve this problem. As is known, the Newton iteration method may be the most popular solver for a general system of nonlinear equations $G(x) = 0$. However, the disadvantages of the Newton iteration method is that, at each iteration step, it requires the explicit form of the $N \times N$ Jacobian matrix $G'(x^{(n)})$, where $x^{(n)}$ denotes the $n$th-approximation to $x$. So the computation of the Newton iteration method could be much more expensive. In order to overcome this disadvantage, Picard iteration method has been used to solve the system \eqref{5.4}, where, for a given iteration $n$, we have:
\begin{equation}
    U^{(n+1)} = K^{-1}_\epsilon R(U^{(n)}),
\end{equation}
until $|| U^{(n+1)} - U^{(n)} || > \epsilon^p$, where $U^{(n)}$ denotes the $n$th-approximation to $U$ and $\epsilon^p$ represents the tolerance of the Picard iteration scheme.
\end{remark}

\subsection{Fast Computation of Hypergeometric Functions}

Accurate and efficient computations of the Gauss hypergeometric function ${}_2 F_1 (a,b;c;z)$ is also fundamental to the developed scheme. From \cite{Pearson2009}, we know that there is no simple answer for this problem, and different methods are optimal for different parameter regimes. When $\Re(c) > \Re(a) > 0$ or $\Re(c) > \Re(b) > 0$, the Gauss-Jacobi quadrature method is effective. As stated in \cite{Abramowitz1965}, when $|\arg(1-z)|<\pi$, we have
\begin{align}\label{5.5}
{}_2 F_1 (a,b;c;z) = \frac{\Gamma(c)}{\Gamma(b)\Gamma(c-b)} \int_0^1 (1-zt)^{-a} w_{b,c}(t) dt,~~~\Re(c) > \Re(b) > 0,
\end{align}
where $w_{b,c}(t) = (1-t)^{c-b-1} t^{b-1}$. The parameters $a$ and $b$ in \eqref{5.5} can be interchanged due to the basic power series definition of the hypergeometric function. Transforming $t \to \frac{\tilde{t} + 1}{2}$, we can obtain that
\begin{align*}
\int_0^1 (1-zt)^{-a} w_{b,c}(t) d\tilde{t} =& \frac{1}{2^{c-1}} \int_{-1}^1 \left(1-\frac{1}{2} z -\frac{1}{2} z\tilde{t}  \right)^{-a} \left(1-\tilde{t} \right)^{c-b-1} \left(1+\tilde{t} \right)^{b-1} dt \\
=& \sum_{j=1}^{Q} w_j^{GJ} \left(1-\frac{1}{2} z -\frac{1}{2} zt_j^{GJ}  \right)^{-a} + E_{Q} (a,b,z),
\end{align*}
where $t_j^{GJ}$ and $w_j^{GJ}$ are the Gauss-Jacobi nodes and weights on $[-1,1]$, and $Q$ is the number of mesh points. Error bounds for this method are discussed in \cite{Gautschi2002}.

\subsection{Algorithm of IMEX($p$) Methods}

We present the main stages of the developed fast IMEX($p$) methods for efficient time-integration of nonlinear FDEs in Algorithm \ref{alg:1}. The algorithm particularly described the IMEX(0) approach, but the main steps remain the same for IMEX(1), with slight modifications regarding the number of sets of correction weights and forms for the nonlinear system matrices. The operators $\mathcal{F}(\cdot)$ and $\mathcal{F}^{-1}(\cdot)$ represent, respectively, the forward and inverse Fast Fourier Transforms.

\renewcommand\thealgorithm{I}
\begin{algorithm}[!htbp]
	\caption{Fast IMEX(0) Integration Scheme for Nonlinear FDEs.}
	\label{alg:1}
	\begin{algorithmic}[1]
	    \STATE{\textbf{Known database:}}
		\STATE{Initial parameters $h$, $N$, $\epsilon$, $\epsilon^p$, $\lambda$, number of correction terms and powers $\sigma, \delta$.}
		\STATE{Compute the first column $a_1$ of $A$ defined in \eqref{eq:A}. Compute $B$ and $C$ defined in \eqref{eq:BC} with sparse allocation.}
		\STATE{Compute corrections using \eqref{2.19}-\eqref{2.20}, \eqref{2.24}, \eqref{eq:corr_force1}-\eqref{eq:corr_force2} and \eqref{3.11} (or \eqref{3.12}) for all $N$ time-steps.}
		\STATE{Compute $D_\varrho$. Assemble the first column $K_{\epsilon,1}$ of $K_\epsilon$ in \eqref{eq:block_circulant}.}
		\STATE{Compute $\Lambda = \mathcal{F}((D_\varrho \otimes I_d) K_{\epsilon,1})$ to obtain \eqref{5.3}.}
		\STATE{\textbf{Picard Iteration:} Initial guess $U^{(0)}$.}
		\WHILE{$e > \epsilon^p$}
		    \STATE{Compute $D(U^{(n)})$ using \eqref{eq:Dvector}. Compute $F(U^{(n)})$.}
		    \STATE{Compute $R(U^{(n)})$ using \eqref{5.1}}
		    \STATE{Compute $r_\epsilon = \mathcal{F}\left((D_\varrho \otimes I_d) R(U^{(n)})\right)$}
		    \STATE{Solve $\tilde{r}_\epsilon = \Lambda^{-1} r_\epsilon$}
		    \STATE{Compute the updated solution vector $U^{(n+1)} = D^{-1}_\varrho \mathcal{F}^{-1}(\tilde{r}_\epsilon)$.}
		    \STATE{Compute $e = ||U^{(n+1)} - U^{(n)}||$.}
		    \STATE{$n = n+1$}		        
		\ENDWHILE
		\RETURN{$U^{(n+1)}$}
	\end{algorithmic}
\end{algorithm}

\begin{remark}\label{rem:FAM}
For the FAMMs \eqref{2.16}, \eqref{3.4} and the FAMMs with correction terms \eqref{2.26}, \eqref{3.8}, we also can use the fast implementation proposed in this section to construct the corresponding fast methods.
\end{remark}

\section{Numerical Tests}
\label{section5}

We present several numerical examples to verify our theoretical analysis presented in the previous sections. 
In all presented numerical examples, we utilize a numerical tolerance $\epsilon = 5 \times 10^{-9}$ for the fast inversion step. For all hypergeometric functions involved in the evaluation of correction weights and history load term, we utilize $Q = 200$ Gauss-Jacobi quadrature points. One exception is the incomplete beta function evaluated for the history load correction in \eqref{2.24}. For this case, the argument $\frac{k}{k+1}$ approaches 1 as $N$ increases, and a numerical quadrature becomes a poor choice due to singularities. In that sense we evaluate the incomplete beta function using the native MATLAB implementation. Furthermore, given $\Omega = (0,\,T]$, we utilize the following quantities:
\[\mbox{err}_N(h) = \frac{\left\| u(t_N) - u_N\right\|_{\infty}}{\left\| u(t_N)\right\|_{\infty}},~~~\mbox{err}(h) = \frac{\max\limits_{0 \le k \le N} \left\| u(t_k) - u_k\right\|_{\infty}}{\max\limits_{0 \le k \le N} \left\| u(t_k)\right\|_{\infty}},\]
\[\mbox{Order}_1 = \log_2 \left[\frac{\mbox{err}_N(h)}{\mbox{err}_N(h/2)}\right],~~~\mbox{Order}_2 = \log_2 \left[\frac{\mbox{err}(h)}{\mbox{err}(h/2)} \right]\]
to denote the error at the endpoint $T$, the global error on the solution interval $\Omega$ and convergence order of the used method at the endpoint $T$ and on the solution interval $\Omega$, respectively. The developed framework was implemented in MATLAB R2019a and was run in a desktop computer with Intel Core i7-6700 CPU with 3.40 GHz, 16 GB RAM and Ubuntu 18.04.2 LTS operating system. 

\begin{example}\label{Sec:Ex1}
Linear FDE (see e.g. \cite{Zayernouri2016MS}):
\begin{equation}\label{eq:linear_convergence}
    {}^C_{0} D_t^\alpha u(t) = f(t),~~\alpha \in (0,1],~t \in (0,1];~~~u(0) = 0.
\end{equation}
The exact solution of \eqref{eq:linear_convergence} is $u(t) = t^{p + \alpha}$ for $p=1,2$. Therefore the corresponding force term is $f(t) = \frac{\Gamma(\alpha + p + 1)}{\Gamma(p + 1)} t^p$. Recalling Remark \ref{rem:FAM}, we can employ the fast inversion scheme directly to the FAMM \eqref{2.16} in order to obtain a fast FAMM. Therefore, in this example we compare the performance between the fast and original FAMMs \eqref{2.16}, where we verify the computational complexity and accuracy of both original and fast schemes. 

Table \ref{tab:Conv_AM_ZM} presents the obtained results for the implemented FAMMs and at the endpoint $T=1$. Similar to the results in \cite{Zayernouri2016MS}, we observe that the convergence order is independent of the fractional order $\alpha$, preserving the accuracy of the integer-order methods. The computational times for the original and fast FAMMs are illustrated in Figure \ref{fig:CompTime_ZM}. We observe the computational complexity of $\mathcal{O}(N \log N)$ for the developed fast FAMMs. Since no break-even point is observed between both methods, the fast method is more computationally efficient regardless of the value of $N$. 
\begin{table}[!htbp]\label{tab:Conv_AM_ZM}
    \footnotesize
	\centering
	\caption{The errors at the endpoint and convergence orders of the FAMMs for \eqref{eq:linear_convergence} with $p = 0$ (upper table) and $p = 1$ (lower table).}
	\begin{tabular}{@{}lllllllll@{}}
	    \multicolumn{9}{c}{\textbf{The fast FAMM with $p = 0$}.} \\
		\toprule
		{}&\multicolumn{2}{c}{$\alpha=0.1$}&{}&\multicolumn{2}{c}{$\alpha=0.5$}&{}&\multicolumn{2}{c}{$\alpha=0.9$} \\
		\cline{2-3}\cline{5-6}\cline{8-9}
		$h$  & $\mbox{err}_N(h)$ & $\mbox{Order}_1$ & {} & $\mbox{err}_N(h)$ & $\mbox{Order}_1$ & {} & $\mbox{err}_N(h)$ & $\mbox{Order}_1$ \\ \midrule
		$2^{-3}$  & 5.2795e--03 &      -- & {} &  9.7878e--03 &     --  & {} & 1.0461e--03 &      -- \\
        $2^{-4}$  & 2.6031e--03 &  1.0202 & {} &  5.0535e--03 &  0.9537 & {} & 6.0576e--04 &  0.7883 \\
        $2^{-5}$  & 1.2912e--03 &  1.0115 & {} &  2.5658e--03 &  0.9779 & {} & 3.3733e--04 &  0.8446 \\
        $2^{-6}$  & 6.4272e--04 &  1.0064 & {} &  1.2925e--03 &  0.9892 & {} & 1.8357e--04 &  0.8778 \\
        $2^{-7}$  & 3.2058e--04 &  1.0035 & {} &  6.4866e--04 &  0.9947 & {} & 9.8414e--05 &  0.8994 \\
        \bottomrule
    \end{tabular}

    \vspace*{0.3cm}
    
	\begin{tabular}{@{}lllllllll@{}}
	    \multicolumn{9}{c}{\textbf{The fast FAMM with $p = 1$}.} \\
		\toprule
		{}&\multicolumn{2}{c}{$\alpha=0.1$}&{}&\multicolumn{2}{c}{$\alpha=0.5$}&{}&\multicolumn{2}{c}{$\alpha=0.9$} \\
		\cline{2-3}\cline{5-6}\cline{8-9}
		$h$  & $\mbox{err}_N(h)$ & $\mbox{Order}_1$ & {} & $\mbox{err}_N(h)$ & $\mbox{Order}_1$ & {} & $\mbox{err}_N(h)$ & $\mbox{Order}_1$ \\ \midrule
		$2^{-3}$   & 2.1934e--05 &      -- &  {} & 4.0975e--04 &      -- & {} & 4.9840e--04 &      -- \\
        $2^{-4}$   & 5.5334e--06 &  1.9869 &  {} & 9.6888e--05 &  2.0804 & {} & 1.2210e--04 &  2.0292 \\
        $2^{-5}$   & 1.4102e--06 &  1.9723 &  {} & 2.3574e--05 &  2.0391 & {} & 2.9761e--05 &  2.0366 \\
        $2^{-6}$   & 3.5857e--07 &  1.9756 &  {} & 5.8150e--06 &  2.0193 & {} & 7.2452e--06 &  2.0383 \\
        $2^{-7}$   & 8.9586e--08 &  2.0009 &  {} & 1.4452e--06 &  2.0085 & {} & 1.7663e--06 &  2.0363 \\
        \bottomrule
    \end{tabular}
\end{table}
\begin{figure}[!htbp]
    \centering
    \includegraphics[width=0.5\textwidth]{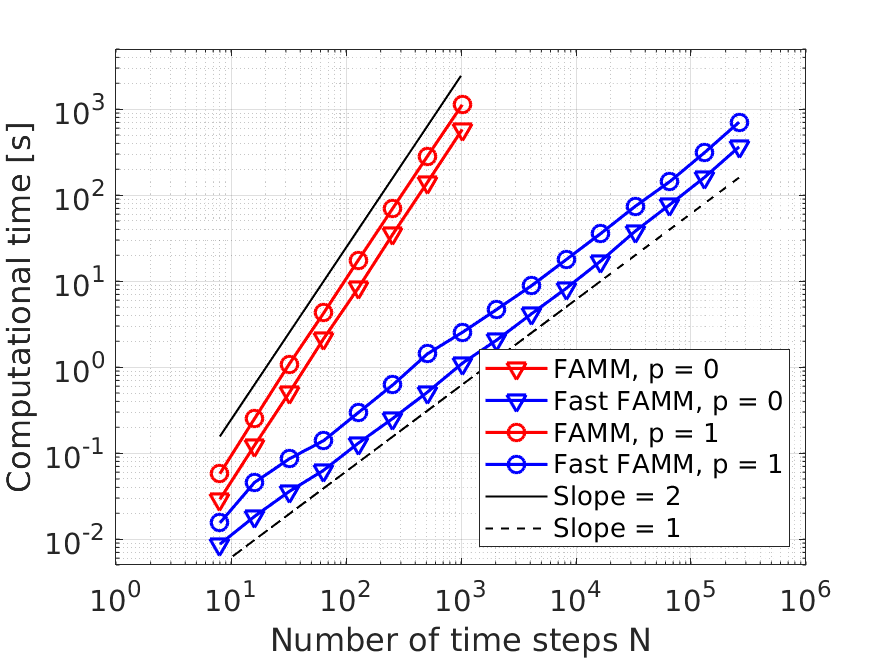}
    \caption{Computational time \textit{versus} number of time steps $N$ for the original (red curves) and the developed, fast (blue curves) FAMMs. \label{fig:CompTime_ZM}}
\end{figure}
\end{example}

\begin{example}\label{Sec:Ex2}
Stiff FDE (see e.g. \cite{Cao2016}):
\begin{equation} \label{eq:SFDE}
    {}^C_{0} D_t^\alpha u (t) = P u(t) + S u(t) + g(t),~~\alpha \in (0,1],~t \in (0,10];~~~u(0) = [1,\,1,\,1]^T,
\end{equation}
where
\begin{equation*}
    P = 
    \begin{bmatrix*}[r]
        -1      &  0      &  0.001 \\
        -0.0005 & -0.0008 & -0.0002 \\
        0.001   &  0      & -0.001
    \end{bmatrix*}, \quad
    S = 
    \begin{bmatrix*}[r]
        -0.006  &  0      &  0.002 \\
        -0.001  & -0.002  &  0     \\
         0      & -0.005  & -0.008
    \end{bmatrix*},
\end{equation*}
\begin{equation*}
    g(t) = \left( {\begin{array}{@{}c@{}}
        a_1 \Gamma_1 t^{\sigma_1 - \alpha} + a_2 \Gamma_2 t^{\sigma_2 - \alpha} \\[1\jot]
        a_3 \Gamma_3 t^{\sigma_3 - \alpha} + a_4 \Gamma_4 t^{\sigma_4 - \alpha} \\[1\jot]
        a_5 \Gamma_5 t^{\sigma_5 - \alpha} + a_6 \Gamma_6 t^{\sigma_6 - \alpha} \\
        \end{array}} \right)
     - \left(P + S\right)
     \left( {\begin{array}{@{}c@{}}
        a_1 t^{\sigma_1} + a_2 t^{\sigma_2} + 1 \\[1\jot]
        a_3 t^{\sigma_3} + a_4 t^{\sigma_4} + 1 \\[1\jot]
        a_5 t^{\sigma_5} + a_6 t^{\sigma_6} + 1 \\
     \end{array}} \right),
\end{equation*}
with $\Gamma_k = \frac{\Gamma(\sigma_k + 1)}{\Gamma(\sigma_k + 1 - \beta)}$, $1 \le k \le 6$. Therefore, the exact solution for the stiff FDE \eqref{eq:SFDE} is given by:
\begin{equation*}
    u(t) = \left( a_1 t^{\sigma_1} + a_2 t^{\sigma_2} + 1,~a_3 t^{\sigma_3} + a_4 t^{\sigma_4} + 1,~a_5 t^{\sigma_5} + a_6 t^{\sigma_6} + 1 \right)^T,
\end{equation*}
where, as in \cite{Cao2016}, we consider $\sigma_1 = \alpha$, $\sigma_2 = 2\alpha$, $\sigma_3 = 1+\alpha$, $\sigma_4 = 5 \alpha$, $\sigma_5 = 2$, $\sigma_6 = 2 + \alpha$, and $a_1 = 0.5$, $a_2 = 0.8$, and $a_3 = a_4 = a_5 = a_6 = 1$. For the numerical solution of \eqref{eq:SFDE}, we take $f(t,u) = S u(t) + g(t)$ and employ the IMEX($p$) scheme with $\alpha = 0.3$, utilizing a Picard iteration tolerance of $\epsilon^p = 5 \times 10^{-7}$. We remark that the coefficients of $P$ and $S$ are taken as small enough values in order to satisfy the Lipschitz condition for the Picard iteration scheme, and nevertheless, the choice of such values still makes \eqref{eq:SFDE} stiff. The obtained results are presented in Table \ref{tab:Conv_Stiff}, where we obtain first-order convergence for the IMEX(0) method without using correction terms. On the other hand, for IMEX(1), we obtain second-order convergence when using $M=3$ correction terms with correction powers $\sigma = \lbrace \alpha,\,2\alpha,\,1+\alpha \rbrace$ and $\delta = \lbrace 2\alpha,\,1+\alpha,\,5\alpha \rbrace$.

\begin{table}[!htbp]\label{tab:Conv_Stiff}
    \footnotesize
	\centering
	\caption{The global errors and convergence orders of the methods IMEX($p$) for \eqref{eq:SFDE} with $p = 0$ (upper table) and $p = 1$ (lower table) and $\alpha = 0.3$.}
	\begin{tabular}{@{}llllll@{}}
	    \multicolumn{6}{c}{\textbf{IMEX(0)}} \\
		\toprule
		{}&\multicolumn{2}{c}{$M=0$}&{}&\multicolumn{2}{c}{$M=1$} \\
		\cline{2-3}\cline{5-6}
		$h$  & $\mbox{err}(h)$ & $\mbox{Order}_2$ & {} & $\mbox{err}(h)$ & $\mbox{Order}_2$ \\ \midrule
		$2^{-3}$  & 1.9340e--02 & --      & {} & 1.9271e--02 &   --  \\
        $2^{-4}$  & 9.7036e--03 &  0.9950 & {} & 9.6793e--03 &  0.9935   \\
        $2^{-5}$  & 4.8614e--03 &  0.9972 & {} & 4.8518e--03 &  0.9964   \\
        $2^{-6}$  & 2.4334e--03 &  0.9984 & {} & 2.4294e--03 &  0.9979   \\
        $2^{-7}$  & 1.2175e--03 &  0.9991 & {} & 1.2157e--03 &  0.9988   \\
        \bottomrule
    \end{tabular}

    \vspace*{0.3cm}
    \setlength{\tabcolsep}{1.3mm}{
	\begin{tabular}{@{}llllllllllll@{}}
		\multicolumn{12}{c}{\textbf{IMEX(1)}} \\
		\toprule
		{}&\multicolumn{2}{c}{$M=0$}&{}&\multicolumn{2}{c}{$M=1$}&{}&\multicolumn{2}{c}{$M=2$}&{}&\multicolumn{2}{c}{$M=3$} \\
		\cline{2-3}\cline{5-6}\cline{8-9}\cline{11-12}
		$h$  & $\mbox{err}(h)$ & $\mbox{Order}_2$ & {} & $\mbox{err}(h)$ & $\mbox{Order}_2$ & {} & $\mbox{err}(h)$ & $\mbox{Order}_2$ & {} & $\mbox{err}(h)$ & $\mbox{Order}_2$ \\ \midrule
		$2^{-3}$  & 6.5717e--04 &   --   & {} & 4.3931e--04 & --      & {} & 2.4945e--04 & --     & {} & 2.1176e--04 & -- \\
        $2^{-4}$  & 3.8296e--04 & 0.7791 & {} & 1.5083e--04 & 1.5424  & {} & 6.2702e--05 & 1.9922 & {} & 5.6167e--05 & 1.9146 \\
        $2^{-5}$  & 2.3992e--04 & 0.6746 & {} & 5.5600e--05 & 1.4397  & {} & 1.5730e--05 & 1.9950 & {} & 1.4606e--05 & 1.9431 \\
        $2^{-6}$  & 1.5407e--04 & 0.6390 & {} & 2.1541e--05 & 1.3680  & {} & 4.0658e--06 & 1.9519 & {} & 3.7593e--06 & 1.9581 \\
        $2^{-7}$  & 9.9312e--05 & 0.6335 & {} & 8.5481e--06 & 1.3334  & {} & 1.7074e--06 & 1.2517 & {} & 9.6316e--07 & 1.9646 \\
        \bottomrule
    \end{tabular}}
\end{table}
\end{example}

\begin{example}\label{Sec:Ex3}
Nonlinear FDE (cf. \cite{Cao2016,Zeng2018}):
\begin{equation}\label{eq:NLFDE}
    {}^C_{0} D_t^\alpha u(t) = \lambda u(t) + f(t,u(t)),~~\alpha \in (0,1],~t \in \Omega;~~~u(0) = u_0.
\end{equation}
We consider the following cases for \eqref{eq:NLFDE}:
\begin{itemize}
    \item \textbf{Case I)} Let $\lambda = -0.2$, $f(t,u(t)) = 0$ and $u_0 = 1$. The corresponding analytical solution is given by $u(t) = E_{\alpha}(-0.2t^\alpha)$, where $E_{\alpha}(t)$ represents the Mittag-Leffler function (cf. \cite{Mainardi2010}).
    
    \item \textbf{Case II)} Let $\lambda = -1$, $f(t,u(t)) =  -0.1 u^2 + g(t)$, $u_0 = 1$ and choose $g(t)$ such that the exact solution of \eqref{eq:NLFDE} is given by $u(t) = 1 + t + t^2/2 + t^3/3 +4^4/4$.
    
    \item \textbf{Case III)} Let $\lambda = -1$, $f(t,u(t)) = 0.01 u\left(1 - u^2\right) + 2 \cos(2\pi t)$ and $u_0 = 1$.
\end{itemize}

We start with \textbf{Case I)}, for which we consider $\Omega = (0, 40]$ and a tolerance $\epsilon^p = 10^{-7}$ for the Picard iteration, with varying number of correction terms $M$ and $\alpha = 0.4$. The obtained results are presented in Table \ref{tab:Ex2_CaseI_FAMMI} for the methods IMEX($p$), where the CPU time (CPU) measured in seconds represent the running time of the methods. We observe the linear convergence rate when using $M = 2$ correction terms for IMEX(0). The convergence rates also improve for IMEX(1), however, we attain the accuracy limit of the scheme when using $M=4$. Such accuracy limit is determined by the value of $\epsilon = 5\times 10^{-9}$ utilized for the fast inversion approach discussed in Section \ref{Sec:FastInversion}.

\begin{table}[!htbp]\label{tab:Ex2_CaseI_FAMMI}
\footnotesize
\caption{The global errors and convergence orders of the methods IMEX($p$) for solving Case I) with $\alpha = 0.4$, varying $h$ and correction terms $M$ with corresponding powers $\sigma_k = \delta_k = k \alpha$.}
\centering
\begin{tabular}{lllllllll}
\toprule
\multirow{2}{*}{$h$} & \multirow{2}{*}{$M$}  & \multicolumn{3}{c}{IMEX(0)} & \multirow{2}{*}{$M$} & \multicolumn{3}{c}{IMEX(1)} \\ \cmidrule(lr){3-5} \cmidrule(l){7-9}   &        & $\mbox{err}(h)$ & $\mbox{Order}_2$      &  $\mbox{CPU}$ &                  & $\mbox{err}(h)$ & $\mbox{Order}_2$   & $\mbox{CPU}$    \\ \midrule
 $2^{-2}$ &  \multirow{5}{*}{0}  & 7.6111e--03 & -- & 0.21 & \multirow{5}{*}{0}  & 5.9001e--03 & -- & 0.37 \\
 $2^{-3}$ &  & 5.7608e--03 & 0.4019 & 0.37 & & 4.6835e--03 & 0.3352 & 0.68 \\
 $2^{-4}$ &  & 4.3450e--03 & 0.4069 & 0.75 & & 3.6810e--03 & 0.3475 & 1.37 \\
 $2^{-5}$ &  & 3.2724e--03 & 0.4090 & 1.35 & & 2.8698e--03 & 0.3591 & 2.65 \\  
 $2^{-6}$ &  & 2.4640e--03 & 0.4093 & 2.65 & & 2.2231e--03 & 0.3684 & 5.32 \\ \midrule
 $2^{-2}$ & \multirow{5}{*}{1} & 1.9406e--03 & --     & 0.22 & \multirow{5}{*}{2}  & 2.1673e--05   & -- & 1.06  \\      
 $2^{-3}$ &  & 1.2144e--03 & 0.6762 & 0.35 &  & 1.0252e--05 & 1.0800 &  2.13         \\
 $2^{-4}$ &  & 7.4541e--04 & 0.7042 & 0.73 &  & 4.7644e--06 & 1.1055 & 4.03           \\
 $2^{-5}$ &  & 4.5062e--04 & 0.7261 & 1.36 &  & 2.1831e--06 & 1.1259 & 7.95           \\ 
 $2^{-6}$ &  & 2.6918e--04 & 0.7433 & 2.66 &  & 9.8961e--07 & 1.1415 & 16.04           \\ \midrule
 $2^{-2}$ & \multirow{5}{*}{2} & 3.2061e--04 & -- & 0.56 & \multirow{5}{*}{4}  & 1.4155e--07 & --  & 1.74 \\
 $2^{-3}$ & & 1.7125e--04 & 0.9047 & 1.10 &   & 5.0125e--08 & 1.4978 & 3.32          \\
 $2^{-4}$ & & 9.0232e--05 & 0.9244 & 1.97 &   & 2.5470e--08 & 0.9767 & 6.49       \\
 $2^{-5}$ & & 4.7019e--05 & 0.9404 & 3.91 &   & 2.7468e--08 & --     & 13.07      \\
 $2^{-6}$ & & 2.4284e--05 & 0.9532 & 7.85 &   & 3.6517e--08 & --     & 26.02       \\ \bottomrule
\end{tabular}
\end{table}

For \textbf{Case II)}, we let $\Omega = (0, 1]$ and $\epsilon^p = 10^{-7}$. The obtained results are presented in Table \ref{tab:Conv_CaseII_IMEX-E}, where we observe that both schemes achieve the theoretical convergence rates for the global error.

\begin{table}[!htbp]\label{tab:Conv_CaseII_IMEX-E}
    \footnotesize
	\centering
	\caption{The global errors and convergence orders of the methods IMEX($p$) for solving Case II) with no correction term for $p = 0$ and $M=2$ correction terms for $p = 1$, with $\sigma = \delta = \lbrace 1-\alpha,\,1\rbrace$.}
	\begin{tabular}{@{}lllllllll@{}}
	    \multicolumn{9}{c}{\textbf{IMEX(0)}} \\
		\toprule
		{}&\multicolumn{2}{c}{$\alpha = 0.2$}&{}&\multicolumn{2}{c}{$\alpha = 0.5$}&{}&\multicolumn{2}{c}{$\alpha = 0.8$} \\
		\cline{2-3}\cline{5-6}\cline{8-9}
		$h$  & $\mbox{err}(h)$ & $\mbox{Order}_2$ & {} & $\mbox{err}(h)$ & $\mbox{Order}_2$& {} & $\mbox{err}(h)$ & $\mbox{Order}_2$ \\ \midrule
		$2^{-6}$  & 1.9855e--02 &   --    & {} & 1.8557e--02 &  --      & {} & 1.7421e--02 & -- \\
        $2^{-7}$  & 9.9111e--03 &  1.0024 & {} & 9.2800e--03 &  0.9998  & {} & 8.7136e--03 & 0.9994   \\
        $2^{-8}$  & 4.9517e--03 &  1.0011 & {} & 4.6425e--03 &  0.9992  & {} & 4.3602e--03 & 0.9989   \\
        $2^{-9}$  & 2.4750e--03 &  1.0005 & {} & 2.3226e--03 &  0.9991  & {} & 2.1821e--03 & 0.9987   \\
        $2^{-10}$ & 1.2373e--03 &  1.0002 & {} & 1.1619e--03 &  0.9992  & {} & 1.0921e--03 & 0.9986   \\
        \bottomrule
    \end{tabular}

    \vspace*{0.3cm}
    
	\begin{tabular}{@{}lllllllll@{}}
	    \multicolumn{9}{c}{\textbf{IMEX(1)}} \\
		\toprule
		{}&\multicolumn{2}{c}{$\alpha = 0.2$}&{}&\multicolumn{2}{c}{$\alpha = 0.5$}&{}&\multicolumn{2}{c}{$\alpha = 0.8$} \\
		\cline{2-3}\cline{5-6}\cline{8-9}
		$h$  & $\mbox{err}(h)$ & $\mbox{Order}_2$ & {} & $\mbox{err}(h)$ & $\mbox{Order}_2$& {} & $\mbox{err}(h)$ & $\mbox{Order}_2$ \\ \midrule
		$2^{-6}$  & 5.3875e--04 & --     & {} & 4.4167e--04  & --      & {} & 3.7370e--04 &      -- \\
        $2^{-7}$  & 1.3610e--04 & 1.9849 & {} & 1.1153e--04  & 1.9856  & {} & 9.4359e--05 &  1.9856   \\
        $2^{-8}$  & 3.4206e--05 & 1.9924 & {} & 2.8033e--05  & 1.9922  & {} & 2.3740e--05 &  1.9909     \\
        $2^{-9}$  & 8.5751e--06 & 1.9960 & {} & 7.0286e--06  & 1.9958  & {} & 5.9679e--06 &  1.9920     \\
        $2^{-10}$ & 2.1498e--06 & 1.9960 & {} & 1.7586e--06  & 1.9988  & {} & 1.4862e--06 &  2.0056     \\
        \bottomrule
    \end{tabular}
\end{table}

For \textbf{Case III)}, we consider $\Omega = (0, 50]$, and we set a numerical tolerance $\epsilon^p = 10^{-6}$ for the Picard iteration. Figures \ref{fig:Zeng_4_1III} and \ref{fig:Zeng_4_1III1} illustrate the obtained highly oscillatory solutions. We also perform a convergence analysis utilizing a benchmark solution with $h = 2^{-11}$ and $M = 3$ correction terms and evaluate the global error. The obtained results are presented in Table \ref{tab:Conv_CaseIII_IMEX-E}, where the expected first- and second-order convergence rates are obtained, respectively, with $p=0$ without correction terms, and $p=1$ using $M=3$ correction terms. The computational times for the developed IMEX schemes are illustrated in Figure \ref{fig:CompTime_Ex2_CaseIII} including the initial phase for computation of correction weights. We observe the computational complexity of $\mathcal{O}(N \log N)$ for the developed schemes even for nonlinear problems, with a small difference between the first- and second-order schemes.

\begin{figure}[!htbp]
\subfigure[$\alpha = 0.2$]{ \centering
\includegraphics[width=0.48\textwidth,height=0.38\textwidth]{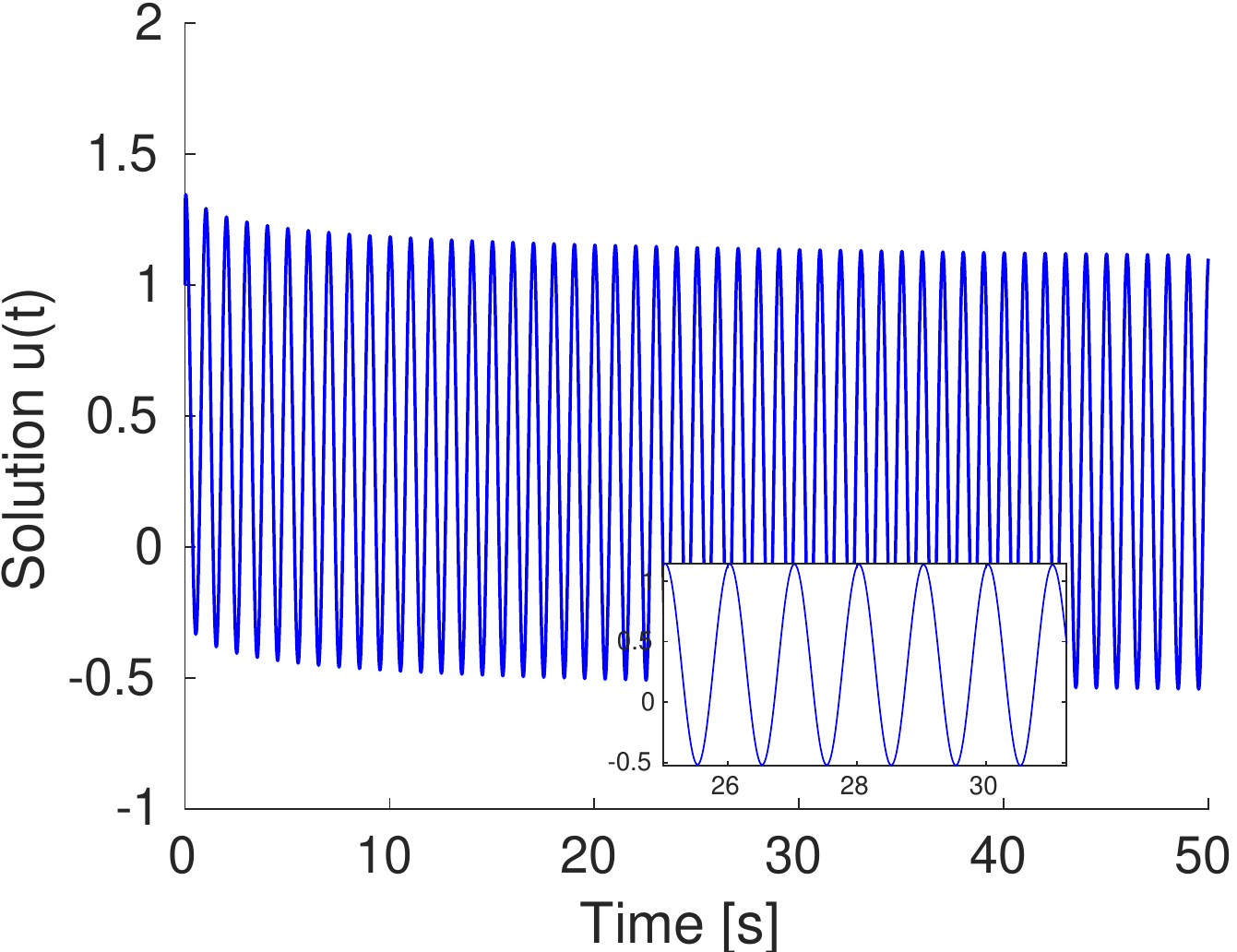}}
\subfigure[$\alpha = 0.5$]{ \centering
\includegraphics[width=0.48\textwidth,height=0.38\textwidth]{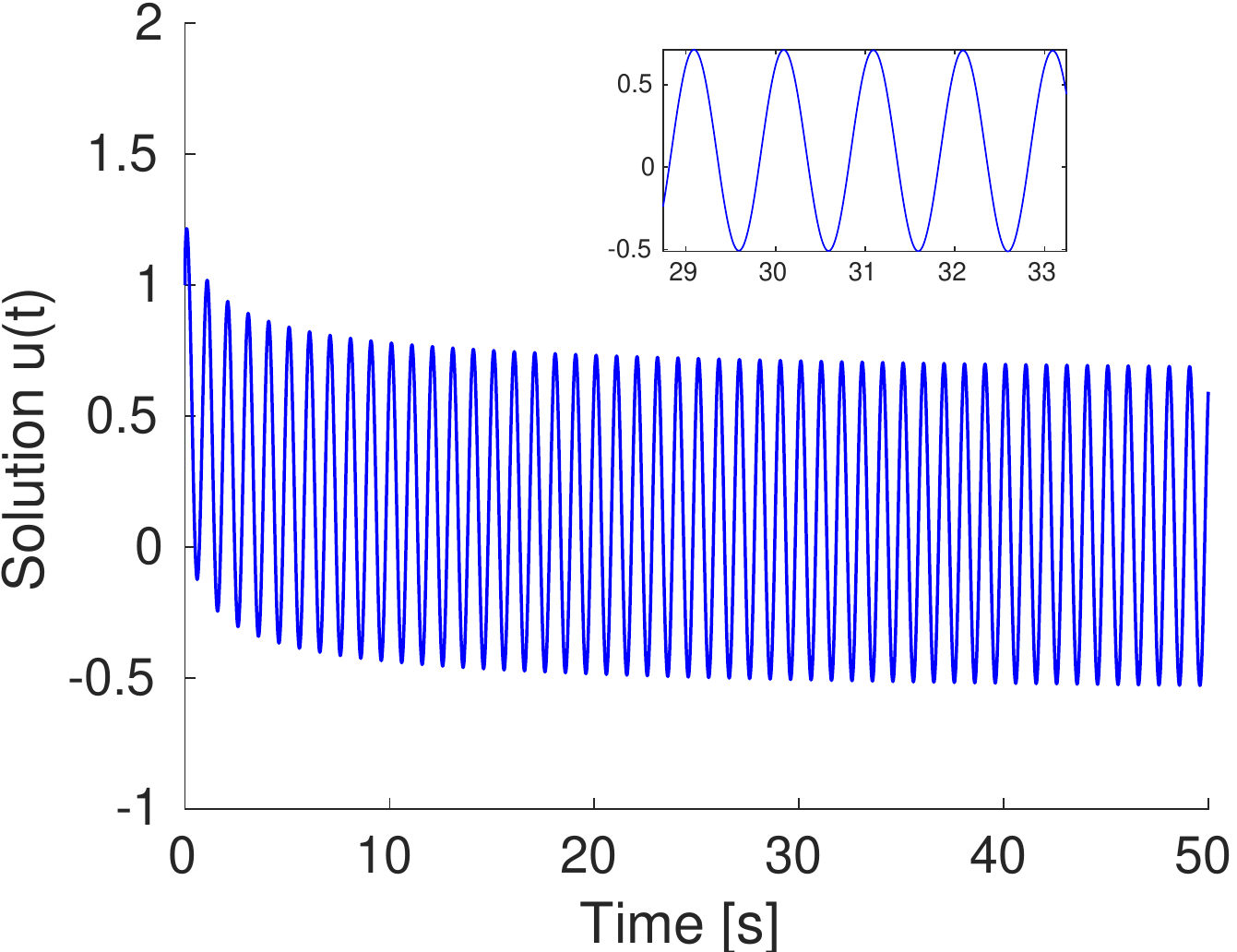}}
\caption{Solution \textit{versus} time $t$ for Case III) using $h = 0.005$ and $M=1$ correction term. \label{fig:Zeng_4_1III}}
\end{figure}

\begin{figure}[!htbp]
\subfigure[$\alpha = 0.7$]{\centering
\includegraphics[width=0.48\textwidth,height=0.38\textwidth]{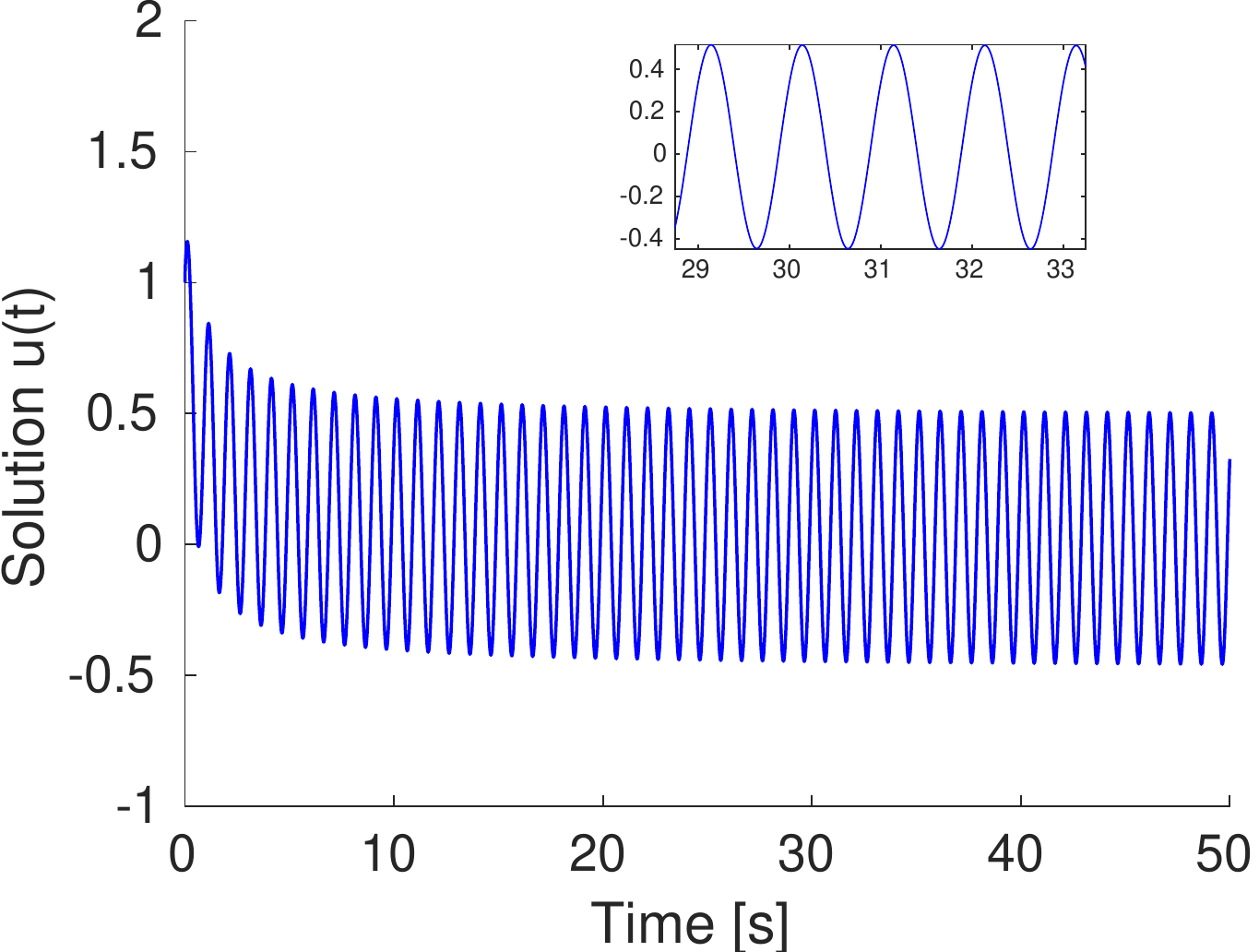}\label{fig:Zeng_4_1III1}}
\subfigure[]{\centering
\includegraphics[width=0.48\textwidth,height=0.38\textwidth]{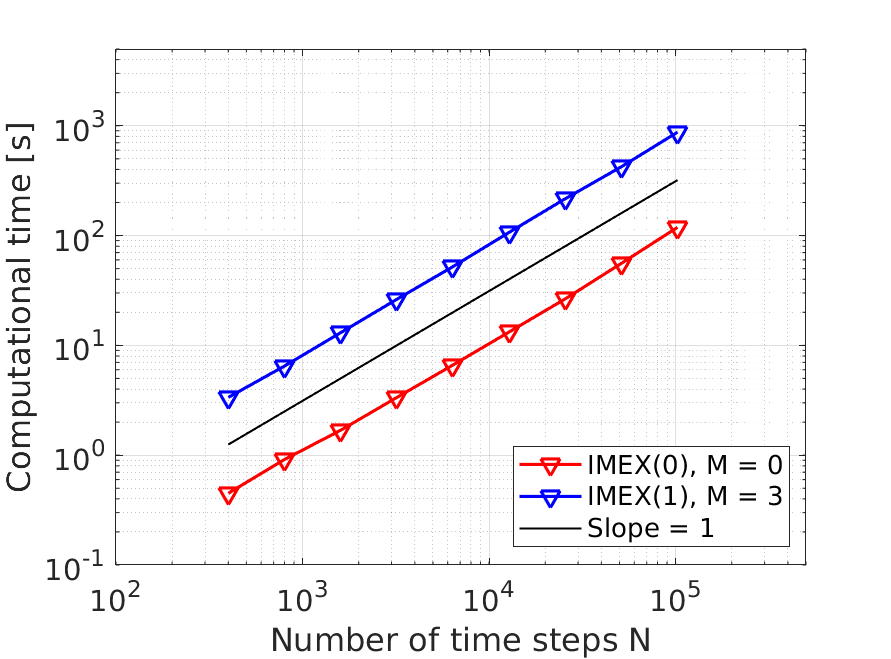}\label{fig:CompTime_Ex2_CaseIII}}
\caption{\textit{(a)} Solution \textit{versus} time $t$ for Case III) using $h = 0.005$ with $M = 1$ correction term. \textit{(b)} Computational time \textit{versus} number of time steps $N$ for the developed IMEX schemes. }
\end{figure}

\begin{table}[!htbp]
    \footnotesize
	\centering
	\caption{Convergence results for the IMEX(p) scheme solving Case III) without corrections for $p = 0$ and $M=3$ correction terms for $p = 1$, with $\sigma = \delta =  \lbrace \alpha,\,2\alpha,\,1+\alpha \rbrace$. \label{tab:Conv_CaseIII_IMEX-E}}
	\begin{tabular}{@{}lllllllll@{}}
	    \multicolumn{9}{c}{\textbf{IMEX(0)}} \\
		\toprule
		{}&\multicolumn{2}{c}{$\alpha = 0.2$}&{}&\multicolumn{2}{c}{$\alpha = 0.5$}&{}&\multicolumn{2}{c}{$\alpha = 0.7$} \\
		\cline{2-3}\cline{5-6}\cline{8-9}
		$h$  & $\mbox{err}(h)$ & $\mbox{Order}_2$ & {} & $\mbox{err}(h)$ & $\mbox{Order}_2$& {} & $\mbox{err}(h)$ & $\mbox{Order}_2$ \\ \midrule
		$2^{-3}$ & 4.4883e--01 &   --    & {} & 3.3552e--01 & --     & {} & 2.5798e--01 &      -- \\
        $2^{-4}$ & 2.2466e--01 &  0.9984 & {} & 1.6681e--01 & 1.0082 & {} & 1.3311e--01 &  0.9546  \\
        $2^{-5}$ & 1.1217e--01 &  1.0021 & {} & 8.3553e--02 & 0.9975 & {} & 6.7629e--02 &  0.9769  \\
        $2^{-6}$ & 5.6354e--02 &  0.9931 & {} & 4.2141e--02 & 0.9875 & {} & 3.4015e--02 &  0.9915  \\
        $2^{-7}$ & 2.8246e--02 &  0.9965 & {} & 2.1163e--02 & 0.9937 & {} & 1.7091e--02 &  0.9929  \\
        \bottomrule
    \end{tabular}

    \vspace*{0.3cm}
    
	\begin{tabular}{@{}lllllllll@{}}
	   \multicolumn{9}{c}{\textbf{IMEX(1)}} \\
		\toprule
		{}&\multicolumn{2}{c}{$\alpha = 0.2$}&{}&\multicolumn{2}{c}{$\alpha = 0.5$}&{}&\multicolumn{2}{c}{$\alpha = 0.7$} \\
		\cline{2-3}\cline{5-6}\cline{8-9}
		$h$  & $\mbox{err}(h)$ & $\mbox{Order}_2$ & {} & $\mbox{err}(h)$ & $\mbox{Order}_2$& {} & $\mbox{err}(h)$ & $\mbox{Order}_2$ \\ \midrule
		$2^{-3}$ & 3.3534e-01 & --      & {} & 1.9664e-01 & --      & {} & 2.5260e-01 &      -- \\
        $2^{-4}$ & 9.5452e-02 &  1.8128 & {} & 8.1541e-02 &  1.2699 & {} & 7.1427e-02 &  1.8223 \\
        $2^{-5}$ & 2.2387e-02 &  2.0921 & {} & 2.0062e-02 &  2.0230 & {} & 2.0503e-02 &  1.8007 \\
        $2^{-6}$ & 5.4892e-03 &  2.0280 & {} & 4.5814e-03 &  2.1306 & {} & 4.8812e-03 &  2.0705 \\
        $2^{-7}$ & 1.8634e-03 &  1.5586 & {} & 1.0402e-03 &  2.1389 & {} & 1.1154e-03 &  2.1296 \\
        \bottomrule
    \end{tabular}
\end{table}
\end{example}

\section{Conclusions}
\label{section6}

We developed two new first- and second-order IMEX schemes for accurate and efficient solution of stiff/nonlinear FDEs with singularities. Both of the schemes are based on the linear multi-step FAMM developed by Zayernouri and Matzavinos \cite{Zayernouri2016MS}, followed by an extrapolation formula from which we obtain the \textit{so-called} IMEX($p$) scheme. In order to handle the inherent singularities of the FDEs, we introduced 4 sets of correction terms for the IMEX($p$) schemes. The convergence and linear stability of the developed schemes is also analyzed. A fast solution for the developed schemes is attained by employing a fast-inversion approach developed by Lu \textit{et al.} \cite{Lux2015} on the resulting nonlinear Toeplitz system, leading to a computational complexity of $\mathcal{O}(N \log N)$. Based on our computational results, we observed that:
\begin{itemize}
    \item When considering a linear problem, the fast implementation of the scheme was significantly faster than the original FAMM by Zayernouri and Matzavinos \cite{Zayernouri2016MS}, without even the presence of a \textit{break-even} point.
    
    \item Both IMEX($p$) schemes achieved global first- (for $p=0$) and second-order (for $p=1$) convergence rates for stiff/nonlinear, highly-oscillatory and singular solutions, given the choice of appropriate sets of correction terms.
    
    \item The computational performance was slightly better for the IMEX(0) scheme. We also remark that such scheme is simpler to implement and generally requires a smaller number of correction terms due to lower regularity requirements to attain first-order accuracy.
\end{itemize}

The main advantages of the developed IMEX schemes in comparison to other works are: larger stability regions when compared to the IMEX schemes developed by Cao \textit{et al.} \cite{Cao2016}; and also a fast solution alternative when compared to the original fractional Adams-Bashforth/Moulton methods developed by Zayernouri and Matzavinos \cite{Zayernouri2016MS}, and the IMEX schemes by Cao \textit{et al.} \cite{Cao2016}. When compared to the matrix-based fast solver for FDEs developed by Lu \cite{Lux2015}, the developed framework in this work handles the numerical solution of nonlinear and singular FDEs instead of only linear ones.


The developed schemes could be used for, \textit{e.g.}, fractional visco-elastic models under complex loading conditions and long-time integration \cite{Jaishankar2013}. Regarding additional constitutive effects, such as fractional visco-elasto-plastic models \cite{Suzuki2016,Suzuki2017Thesis}, and plasticity-driven damage formulations \cite{Suzuki2014}, the developed methods could be potentially applied under simple monotone loads. Furthermore, the introduction of additional sets of correction terms motivates the use of data-infused self-singularity-capturing approaches \cite{Suzuki2018SC}, which would decrease the number of correction terms per set.

\appendix
\section{Discretization Coefficients for the History Load Term}
We present the coefficients $\gamma_{k,j}^{(p)}~(0 \le j \le k,~0 \le k \le N,~p=0,1)$ defined in \eqref{2.15*}:
\label{appendixA}
\[\gamma_{0,0}^{(0)} = 0,~~~\gamma_{0,0}^{(1)} = 0,~~~\gamma_{1,0}^{(1)} = \frac{\mathcal{A}_{1,1} - \mathcal{A}_{1,0}}{h},\]
\[\gamma_{1,1}^{(1)} = \frac{\mathcal{A}_{1,0} - \mathcal{A}_{1,1}}{h},~~~\gamma_{2,0}^{(1)} = \frac{-2\mathcal{A}_{2,0} + \mathcal{A}_{2,1} - \mathcal{A}_{2,2}}{2h} + \frac{\mathcal{B}_{2,1} - \mathcal{B}_{2,2}}{(2-\alpha)h^2},\]
\[\gamma_{2,1}^{(1)} = \frac{\mathcal{A}_{2,0} - \mathcal{A}_{2,1} + 2\mathcal{A}_{2,2}}{h} - \frac{2\mathcal{B}_{2,1} - 2\mathcal{B}_{2,2}}{(2-\alpha)h^2},~~~\gamma_{2,2}^{(1)} = \frac{ \mathcal{A}_{2,1} - 3\mathcal{A}_{2,2}}{2h} + \frac{\mathcal{B}_{2,1} - \mathcal{B}_{2,2}}{(2-\alpha)h^2},\]
for $k \ge 1$,
\begin{align*}
\gamma_{k,j}^{(0)} = \left\{ {\begin{array}{*{20}{l}}
\dfrac{\mathcal{A}_{k,1} - \mathcal{A}_{k,0}}{h},&j=0,\\[2\jot]
\dfrac{\mathcal{A}_{k,j-1} - 2\mathcal{A}_{k,j} + \mathcal{A}_{k,j+1}}{h},~~~& 1\le j \le k-1,\\[2\jot]
\dfrac{\mathcal{A}_{k,k-1} - \mathcal{A}_{k,k}}{h},&j=k,
\end{array}} \right.
\end{align*}
and, for $k \ge 3$,
\begin{align*}
\gamma_{k,j}^{(1)} = \left\{ {\begin{array}{*{20}{l}}
\dfrac{-2\mathcal{A}_{k,0} + \mathcal{A}_{k,1} - \mathcal{A}_{k,2}}{2h} + \dfrac{\mathcal{B}_{k,1} - \mathcal{B}_{k,2}}{(2-\alpha)h^2},&j=0,\\[3\jot]
\dfrac{2\mathcal{A}_{k,0} \!-\! 2\mathcal{A}_{k,1} \!+\! 3\mathcal{A}_{k,2} \!-\! \mathcal{A}_{k,3}}{2h} \!+\! \dfrac{-2\mathcal{B}_{k,1} \!+\! 3\mathcal{B}_{k,2} \!-\! \mathcal{B}_{k,3}}{(2-\alpha)h^2},&j=1,\\[3\jot]
\dfrac{\mathcal{A}_{k,j-1} - 3\mathcal{A}_{k,j} + 3\mathcal{A}_{k,j+1} - \mathcal{A}_{k,j+2}}{2h} \\[3\jot]
~~~+ \dfrac{\mathcal{B}_{k,j-1} - 3\mathcal{B}_{k,j} + 3\mathcal{B}_{k,j+1} - \mathcal{B}_{k,j+2}}{(2-\alpha)h^2},& 2\le j \le k-2,\\[3\jot]
\dfrac{\mathcal{A}_{k,k-2} \!-\! 3\mathcal{A}_{k,k-1} \!+\! 4\mathcal{A}_{k,k}}{2h} \!+\! \dfrac{\mathcal{B}_{k,k-2} \!-\! 3\mathcal{B}_{k,k-1} \!+\! 2\mathcal{B}_{k,k}}{(2-\alpha)h^2},~~~&j=k-1,\\[3\jot]
\dfrac{\mathcal{A}_{k,k-1} - 3\mathcal{A}_{k,k}}{2h} + \dfrac{\mathcal{B}_{k,k-1} - \mathcal{B}_{k,k}}{(2-\alpha)h^2},&j=k.
\end{array}} \right.
\end{align*}

\section{Proofs}
\label{appendix}

\subsection{Proof of Lemma \ref{lemma2.3}}
Before proof of Lemma \ref{lemma2.3}, we need some preparatory results by introducing the notations:
\begin{align}\label{A.1}
a_l^{(\alpha)} = (l+\theta+1)^{1-\alpha} - (l+\theta)^{1-\alpha},~~~l \ge 0,
\end{align}
\begin{align}\label{A.2}
b_l^{(\alpha)} = \frac{(l+\theta+1)^{2-\alpha} - (l+\theta)^{2-\alpha}}{2-\alpha} - \frac{(l+\theta+1)^{1-\alpha} + (l+\theta)^{1-\alpha}}{2},~~~l \ge 0,
\end{align}
and
\begin{align}\label{A.3}
c_l^{(\alpha)} = \left\{ {\begin{array}{*{20}{l}}
a_0^{(\alpha)} + b_0^{(\alpha)},&l=0,\\[3\jot]
a_l^{(\alpha)} + b_l^{(\alpha)} - b_{l-1}^{(\alpha)},~~~& 1\le l \le k-2,\\[3\jot]
a_l^{(\alpha)} - b_{l-1}^{(\alpha)},&l=k-1.
\end{array}} \right.
\end{align}
where $\theta \in [0,1]$. The following lemma states the properties of the above defined notations.

\begin{lemma}\label{lemmaA.1}
For any $\alpha~(0 < \alpha \le 1)$ and $\big\{ a_l^{(\alpha)} \big\}$, $\big\{ b_l^{(\alpha)} \big\}$ and $\big\{ c_l^{(\alpha)} \big\}$ defined in \eqref{A.1}-\eqref{A.3}, respectively, it holds that
\begin{itemize}
\item $a_0^{(\alpha)} > a_1^{(\alpha)} > a_2^{(\alpha)} > \cdots > a_l^{(\alpha)} > 0~\mbox{as}~l \rightarrow \infty,~a_0^{(\alpha)} \le v_0$;
\vskip 0.1 cm
\item $b_0^{(\alpha)} > b_1^{(\alpha)} > b_2^{(\alpha)} > \cdots > b_l^{(\alpha)} > 0~\mbox{as}~l \rightarrow \infty$;
\vskip 0.1 cm
\item $c_2^{(\alpha)} > c_3^{(\alpha)} > c_4^{(\alpha)} > \cdots > c_{k-1}^{(\alpha)} > 0,~\left| c_l^{(\alpha)} \right| \le v_0~(l=0~\mbox{or}~1),~c_2^{(\alpha)} \le v_0$,
\end{itemize}
where $v_0 > 0$ is a constant.
\end{lemma}

\begin{proof}
From the definition of $a_l^{(\alpha)}$,  we can verify that $a_0^{(\alpha)} = (\theta+1)^{1-\alpha} - \theta^{1-\alpha}$ can be bounded by a constant $v_1 > 0$ and
\begin{align*}
a_l^{(\alpha)} = (1-\alpha) \int_{l}^{l+1} (x+\theta)^{-\alpha} dx,~~~l = 0,1,2,\ldots,
\end{align*}
where $(x+\theta)^{-\alpha} > 0$ is a monotone decreasing function. Then it is no difficult to verify that
$$a_0^{(\alpha)} > a_1^{(\alpha)} > a_2^{(\alpha)} > \cdots > a_l^{(\alpha)} > 0.$$
For the second conclusion and the first part of the third conclusion, the proofs are similar to that for Lemma 2.1 and Lemma 2.2 in \cite{Gao2014}. Besides, in view of the definition \eqref{A.3} of $c_l^{(\alpha)}$, we have that
\begin{align*}
& \left| c_0^{(\alpha)} \right| = \left| \frac{(\theta+1)^{1-\alpha} - 3\theta^{1-\alpha}}{2} + \frac{(\theta+1)^{2-\alpha} - \theta^{2-\alpha}}{2-\alpha} \right| \nonumber\\
= & \left| \frac{(\theta+1)^{1-\alpha} - 3\theta^{1-\alpha}}{2} + \frac{(2-\alpha)(\theta+\xi)^{1-\alpha}}{2-\alpha} \right| \le \frac{3\left[ (\theta+1)^{1-\alpha} + \theta^{1-\alpha}\right]}{2},
\end{align*}
where $\xi \in (0,1)$ and the mean value theorem has been used. It means $\left| c_0^{(\alpha)} \right|$ can be bounded by a constant $v_2 > 0$. Similarly, we can get there exists a constant $v_3 > 0$, such that
\begin{align*}
\left| c_1^{(\alpha)} \right| \le \frac{3(\theta+2)^{1-\alpha} + 4(\theta+1)^{1-\alpha} + \theta^{1-\alpha}}{2} \le v_3.
\end{align*}
For $c_2^{(\alpha)}$, it follows
\begin{align*}
&c_2^{(\alpha)} = \frac{(\theta+3)^{1-\alpha} \!-\! 2(\theta+2)^{1-\alpha} \!+\! (\theta+1)^{1-\alpha}}{2} \!+\! \frac{(\theta+3)^{2-\alpha} \!-\! 2(\theta\!+\!2)^{2-\alpha} \!+\! (\theta\!+\!1)^{2-\alpha}}{2-\alpha} \nonumber\\
= & \frac{(\theta+3)^{1-\alpha} - 2(\theta+2)^{1-\alpha} + (\theta+1)^{1-\alpha}}{2} + \frac{(2-\alpha)(\theta+\eta_2)^{1-\alpha} - (2-\alpha)(\theta+\eta_1)^{1-\alpha}}{2-\alpha} \nonumber\\
\le& \frac{3(\theta+3)^{1-\alpha} - 2(\theta+2)^{1-\alpha} - (\theta+1)^{1-\alpha}}{2} \le v_4,
\end{align*}
where $\eta_1 \in (1,2),\eta_2 \in (2,3)$ and $v_4 > 0$ is a constant. Hence, by setting $v_0 = \max\limits_{1 \le l \le 4} v_l$ and summarizing the above results, all this completes the proof.
~~~\end{proof}

Next, we present the proof for Lemma \ref{lemma2.3}.
\begin{proof}
Firstly, we consider the situation of $p=0$. From \eqref{2.15*}, when $j=0$, we have that
\begin{align*}
&\gamma_{k,0}^{(0)} = \frac{\mathcal{A}_{k,1} - \mathcal{A}_{k,0}}{h} = \frac{1}{h} \int_{t_k}^{t_{k+1}} (t_{k+1}-v)^{\alpha-1} \left[(v-t_1)^{1-\alpha} - (v-t_0)^{1-\alpha}\right] dv \nonumber\\
= & \frac{(t_{k+\tilde{\theta}}-t_1)^{1-\alpha} - (t_{k+\tilde{\theta}}-t_0)^{1-\alpha}}{h} \int_{t_k}^{t_{k+1}} (t_{k+1}-v)^{\alpha-1} dv = -\frac{a_{k-1}^{(\alpha)}}{\alpha} < 0,
\end{align*}
where $t_{k+\tilde{\theta}} = (k+\tilde{\theta}) h,~\tilde{\theta} \in [0,1]$ and the first mean value theorem for integrals has been used. For $j=1,2,\ldots,k-1$, one has
\begin{align*}
& \gamma_{k,j}^{(0)} =\frac{\mathcal{A}_{k,j-1} - 2\mathcal{A}_{k,j} + \mathcal{A}_{k,j+1}}{h} \nonumber\\
= & \frac{1}{h} \int_{t_k}^{t_{k+1}} (t_{k+1}-v)^{\alpha-1} \left[(v-t_{j-1})^{1-\alpha} - 2(v-t_j)^{1-\alpha} + (v-t_{j+1})^{1-\alpha}\right] dv \nonumber\\
= & \frac{(t_{k+\tilde{\theta}}-t_{j-1})^{1-\alpha} - 2(t_{k+\tilde{\theta}}-t_j)^{1-\alpha} + (t_{k+\tilde{\theta}}-t_{j+1})^{1-\alpha}}{h}\int_{t_k}^{t_{k+1}} (t_{k+1}-v)^{\alpha-1}dv \nonumber\\
= & \frac{a_{k-j}^{(\alpha)} - a_{k-j-1}^{(\alpha)}}{\alpha} < 0.
\end{align*}
For $j=k$, it holds that
\begin{align*}
& \gamma_{k,k}^{(0)} = \frac{\mathcal{A}_{k,k-1} - \mathcal{A}_{k,k}}{h} = \frac{1}{h} \int_{t_k}^{t_{k+1}} (t_{k+1}-v)^{\alpha-1} \left[(v-t_{k-1})^{1-\alpha} - (v-t_{k})^{1-\alpha}\right] dv \nonumber\\
= & \frac{(t_{k+\tilde{\theta}}-t_{k-1})^{1-\alpha} - (t_{k+\tilde{\theta}}-t_{k})^{1-\alpha}}{h} \int_{t_k}^{t_{k+1}} (t_{k+1}-v)^{\alpha-1} dv = \frac{a_0^{(\alpha)}}{\alpha} > 0,
\end{align*}
and $\gamma_{k,k}^{(0)} = \dfrac{a_0^{(\alpha)}}{\alpha} \le \dfrac{v_0}{\alpha}$, where Lemma \ref{lemmaA.1} has been used.
Similarly, for $\gamma_{k,j}^{(1)}$, by using the first mean value theorem for integrals, we obtain
\begin{align*}
\gamma_{k,j}^{(1)} = \left\{ {\begin{array}{*{20}{l}}
-\dfrac{c_{k-1}^{(\alpha)}}{\alpha},&j=0,\\[3\jot]
\dfrac{-c_{k-j-1}^{(\alpha)} + c_{k-j}^{(\alpha)}}{\alpha},~~~& 1\le j \le k-1,\\[3\jot]
\dfrac{c_{0}^{(\alpha)}}{\alpha},&j=k.
\end{array}} \right.
\end{align*}
This, together with Lemma \ref{lemmaA.1}, imply that $\gamma_{k,j}^{(1)} < 0~(0 \le j \le k-3)$ and
\[\left| \gamma_{k,k-2}^{(1)} \right| = \left| \frac{-c_{1}^{(\alpha)} + c_{2}^{(\alpha)}}{\alpha} \right| \le \frac{2 v_0}{\alpha},\]
\[\left| \gamma_{k,k-1}^{(1)} \right| = \left| \frac{-c_{0}^{(\alpha)} + c_{1}^{(\alpha)}}{\alpha} \right| \le \frac{2 v_0}{\alpha},~~~\left|\gamma_{k,k}^{(1)} \right| = \left| \frac{c_{0}^{(\alpha)}}{\alpha}\right| \le \frac{v_0}{\alpha}.\]
Suppose that $u(t) = 1$ for $t \in [0,T]$, it follows from \eqref{2.15*} that
\begin{align*}
\mathcal{H}^k (t_{k+1}) = \frac{1}{\Gamma(\alpha) \Gamma(2-\alpha)} \sum_{j=0}^{k} \gamma_{k,j}^{(0)} = \frac{1}{\Gamma(\alpha) \Gamma(2-\alpha)} \sum_{j=0}^{k} \gamma_{k,j}^{(1)} = 0,
\end{align*}
so we get $\sum\limits_{j=0}^k \gamma_{k,j}^{(0)} = 0$ and $\sum\limits_{j=0}^k \gamma_{k,j}^{(1)} = 0$. Hence the lemma is proven.
~~~\end{proof}

\subsection{Proof of Lemma \ref{lemma2.1}}

\begin{proof}
It follows from \eqref{2.11} that
\begin{align*}
&\bigg| {}_{t_k}I_t^\alpha u(t) \Big|_{t=t_{k+1}} - h^\alpha \sum_{j=0}^p \beta_j^{(p)} u(t_{k+1-j}) \bigg| \nonumber\\
= & \frac{1}{\Gamma(\alpha)} \bigg| \int_{t_k}^{t_{k+1}} (t_{k+1}-v)^{\alpha-1} \bigg[u(v) - \sum_{j=0}^p u(t_{k+1-j}) \prod_{i=0,i \ne j}^p \frac{v - t_{k+1-i}}{t_{k+1-j} - t_{k+1-i}} \bigg] dv \bigg| \nonumber\\
\le & \frac{1}{\Gamma(\alpha)} \bigg| \int_{t_k}^{t_{k+1}} (t_{k+1}-v)^{\alpha-1} u^{(p+1)}(\varsigma) \prod_{i=0}^p (v-t_{k+1-i}) dv \bigg| \nonumber\\
\le & \frac{h^{p+1} |u^{(p+1)}(\varsigma)|}{\Gamma(\alpha)} \left| \int_{t_k}^{t_{k+1}} (t_{k+1}-v)^{\alpha-1} dv \right| \le v_5 h^{\alpha+p+1} t_{k+1}^{\sigma-p-1},
\end{align*}
where $\varsigma \in (t_k,t_{k+1})$ and $v_5 > 0$ is a constant independent of $h$. This completes the proof.
\end{proof}

\subsection{Proof of Lemma \ref{lemma2.2}}

\begin{proof}
By \eqref{2.12}, we have for $k = 2,3,\ldots,N-1,~N \ge 3$ (the cases for $k = 0,1$ are easy to check, so we omit these case here) that
\begin{align*}
&\left|\mathcal{H}^k (t_{k+1}) - \mathcal{H}_0^k (t_{k+1})\right| \nonumber\\
=& \frac{1}{\Gamma(\alpha) \Gamma(1-\alpha)} \bigg| \int_{t_k}^{t_{k+1}} \frac{1}{(t_{k+1}-v)^{1-\alpha}} \sum_{j=0}^{k-1} \int_{t_j}^{t_{j+1}} \frac{u'(s) - (\mit\Pi_{1,j} u(s))'}{(v-s)^{\alpha}} ds dv \bigg| \nonumber\\
\le & \frac{1}{\Gamma(\alpha) \Gamma(1-\alpha)} \bigg| \int_{t_k}^{t_{k+1}} \frac{1}{(t_{k+1}-v)^{1-\alpha}} \int_{0}^{t_{1}} \frac{u'(s) - (\mit\Pi_{1,0} u(s))'}{(v-s)^{\alpha}} ds dv \bigg| \nonumber\\
& + \frac{1}{\Gamma(\alpha) \Gamma(1-\alpha)} \bigg| \int_{t_k}^{t_{k+1}} \frac{1}{(t_{k+1}-v)^{1-\alpha}} \sum_{j=1}^{\hat{k} - 1} \int_{t_j}^{t_{j+1}} \frac{u'(s) - (\mit\Pi_{1,j} u(s))'}{(v-s)^{\alpha}} ds dv \bigg| \nonumber\\
& + \frac{1}{\Gamma(\alpha) \Gamma(1-\alpha)} \bigg| \int_{t_k}^{t_{k+1}} \frac{1}{(t_{k+1}-v)^{1-\alpha}} \sum_{j=\hat{k}}^{k-1} \int_{t_j}^{t_{j+1}} \frac{u'(s) - (\mit\Pi_{1,j} u(s))'}{(v-s)^{\alpha}} ds dv \bigg| \nonumber\\
:= & \frac{1}{\Gamma(\alpha) \Gamma(1-\alpha)} \left(I_1 + I_2 + I_3 \right),
\end{align*}
where $\hat{k} \in (1,k)$. For $I_1$, by using the integration by parts, one gets that
\begin{align*}
& I_1 = \left| \int_{t_k}^{t_{k+1}} (t_{k+1}-v)^{\alpha-1} \int_{0}^{t_{1}} (v-s)^{-\alpha} d \big[u(s) - \mit\Pi_{1,0} u(s)\big] dv \right| \nonumber\\
= & \alpha \left| \int_{t_k}^{t_{k+1}} (t_{k+1}-v)^{\alpha-1} \int_{0}^{t_{1}} (v-s)^{-\alpha-1} \left[u(s) - \frac{s-t_1}{- t_1}u(0) - \frac{s}{t_1}u(t_1) \right] ds dv \right| \nonumber\\
= & \alpha\left| \int_{t_k}^{t_{k+1}} (t_{k+1}-v)^{\alpha-1} \int_{0}^{t_{1}} (v-s)^{-\alpha-1} \left[\frac{s-t_1}{- t_1} \int_{0}^s u'(\tau) d\tau - \frac{s}{t_1} \int_{s}^{t_1} u'(\tau) d\tau \right] ds dv \right| \nonumber\\
= & \alpha \sigma\left| \int_{t_k}^{t_{k+1}} (t_{k+1}\!-\!v)^{\alpha-1} \int_{0}^{t_{1}} (v-s)^{-\alpha-1} \left[\frac{s-t_1}{- t_1} \int_{0}^s \tau^{\sigma - 1} d\tau - \frac{s}{t_1} \int_{s}^{t_1} \tau^{\sigma - 1} d\tau \right] ds dv \right| \nonumber\\
\le & \alpha \Bigg| \int_{t_k}^{t_{k+1}} (t_{k+1} - v)^{\alpha-1} \int_{0}^{t_{1}} (v-s)^{-\alpha-1} s^{\sigma} ds dv \Bigg| \nonumber\\
&+ \alpha \Bigg| \int_{t_k}^{t_{k+1}} (t_{k+1} - v)^{\alpha-1} \int_{0}^{t_{1}} (v-s)^{-\alpha-1} (t_1^\sigma - s^{\sigma}) ds dv \Bigg|\nonumber\\
\le & 3\alpha t_1^{\sigma+1} (t_k-t_1)^{-\alpha-1}\left| \int_{t_k}^{t_{k+1}} (t_{k+1}-v)^{\alpha-1} dv \right| = 3 h^{\sigma+\alpha+1} t_{k - 1}^{-\alpha-1}.
\end{align*}
On the other hand, for $I_2$, by using the mean value theorem and the Euler-Maclaurin formula, we arrive at
\begin{align*}
&I_2 = \bigg| \int_{t_k}^{t_{k+1}} (t_{k+1}-v)^{\alpha-1} \sum_{j=1}^{\hat{k}-1} \int_{t_j}^{t_{j+1}} \frac{2s-t_j-t_{j+1}}{2(v-s)^{\alpha}} u''(\xi_j) ds dv \bigg| \nonumber\\
\le &\sigma(\sigma-1)h \bigg| \int_{t_k}^{t_{k+1}} (t_{k+1}-v)^{\alpha-1} \sum_{j=1}^{\hat{k}-1} t_j^{\sigma-2} \int_{t_j}^{t_{j+1}} (v-s)^{-\alpha} dsdv \bigg| \nonumber\\
= & \frac{\sigma(\sigma-1)h}{1-\alpha} \bigg| \int_{t_k}^{t_{k+1}} (t_{k+1}-v)^{\alpha-1} \sum_{j=1}^{\hat{k}-1} t_j^{\sigma-2} \left[(v-t_j)^{1-\alpha} - (v-t_{j+1})^{1-\alpha}\right]dv \bigg| \nonumber\\
\le & \sigma(\sigma-1)h \bigg| \int_{t_k}^{t_{k+1}} (t_{k+1}-v)^{\alpha-1} \sum_{j=1}^{\hat{k}-1} t_j^{\sigma-2} (v-t_{j+1})^{-\alpha}dv \bigg| \nonumber\\
\le & \sigma(\sigma-1)h \sum_{j=1}^{\hat{k}-1} t_j^{\sigma-2} (t_k-t_{j+1})^{-\alpha} \left| \int_{t_k}^{t_{k+1}} (t_{k+1}-v)^{\alpha-1} dv \right| \nonumber\\
\le & \frac{\sigma(\sigma-1)}{\alpha} h^{\alpha+1} t_k^{\sigma-\alpha-1},
\end{align*}
where $\xi_j \in (t_j,t_{j+1})~(j=1,2,\ldots,\hat{k} - 1)$. Next, it holds that
\begin{align*}
&I_3 = \bigg| \int_{t_k}^{t_{k+1}} (t_{k+1}-v)^{\alpha-1} \sum_{j=\hat{k}}^{k-1} \int_{t_j}^{t_{j+1}} \frac{2s-t_j-t_{j+1}}{2(v-s)^{\alpha}} u''(\xi_j) ds dv \bigg| \nonumber\\
\le &\sigma(\sigma-1)h \bigg| \int_{t_k}^{t_{k+1}} (t_{k+1}-v)^{\alpha-1} \sum_{j=\hat{k}}^{k-1} t_j^{\sigma-2} \int_{t_j}^{t_{j+1}} (v-s)^{-\alpha} dsdv \bigg| \nonumber\\
= &\frac{\sigma(\sigma-1)t_{\hat{k}}^{\sigma - 2}h}{1-\alpha} \left| \int_{t_k}^{t_{k+1}} (t_{k+1}-v)^{\alpha-1} \left[(v-t_{\hat{k}})^{1-\alpha} - (v-t_{k})^{1-\alpha}\right] dv \right| \nonumber\\
= & \frac{\sigma(\sigma-1)t_{\hat{k}}^{\sigma - 2}h}{1-\alpha} \left[(t_{k+\hat{\theta}}-t_{\hat{k}})^{1-\alpha} - (t_{k+\hat{\theta}}-t_{k})^{1-\alpha}\right] \left| \int_{t_k}^{t_{k+1}} (t_{k+1}-v)^{\alpha-1} dv \right| \nonumber\\
\le & \frac{\sigma(\sigma-1)t_{\hat{k}}^{\sigma - 2}h^{1+\alpha}}{\alpha(1-\alpha)} \left[(t_{k+\hat{\theta}}-t_{\hat{k}})^{1-\alpha} - (t_{k+\hat{\theta}}-t_{k})^{1-\alpha}\right] \nonumber\\
\le & \frac{\sigma(\sigma-1) T}{\alpha(1-\alpha)} \left[\frac{1}{(k-\hat{k}+\hat{\theta})^\alpha} - \frac{1}{\hat{\theta}^\alpha}\right]h t_{\hat{k}}^{\sigma - 2},
\end{align*}
where $\hat{\theta} \in [0,1]$, $\xi_j \in (t_j,t_{j+1})~(j=\hat{k},\hat{k}+1,\ldots,k-1)$. Then for a suitable $\hat{k}$, there exists a constant $v_6 > 0$ independent of $h$ such that
\begin{align*}
\left|\mathcal{H}^k (t_{k+1}) - \mathcal{H}_0^k (t_{k+1})\right| \le & \frac{1}{\Gamma(\alpha) \Gamma(1-\alpha)} \Bigg\{3 h^{\sigma+\alpha+1} t_{k - 1}^{-\alpha-1} + \frac{\sigma(\sigma-1)}{\alpha} h^{\alpha+1} t_k^{\sigma-\alpha-1} \nonumber\\
& + \frac{\sigma(\sigma-1) T}{\alpha(1-\alpha)} \left[\frac{1}{(k-\hat{k}+\hat{\theta})^\alpha} - \frac{1}{\hat{\theta}^\alpha} \right] h t_{\hat{k}}^{\sigma - 2} \Bigg\} \nonumber\\
\le & v_6\left(h^{\sigma+\alpha+1} t_{k + 1}^{-\alpha-1} + h^{\alpha+1} t_{k+1}^{\sigma-\alpha-1} + h\right).
\end{align*}
Similarly, we can get there exists a constant $v_7 > 0$ independent of $h$ such that
\begin{align*}
\left|\mathcal{H}^k (t_{k+1}) - \mathcal{H}_{1}^k (t_{k+1})\right| \le v_7 \left(h^{\sigma+\alpha+1} t_{k + 1}^{-\alpha-1} + h^{\alpha+2} t_{k+1}^{\sigma-\alpha-2} + h^2 \right).
\end{align*}
Therefore, when setting $C_3 = \max\left\{v_7,v_8 \right\}$, the lemma is proved.
\end{proof}

\subsection{Proof of Lemma \ref{lemma2.4}}

\begin{proof}
From Lemma \ref{lemma2.1} we know that
\begin{align*}
\frac{h^\alpha t_k^{\sigma_r}}{\Gamma(\alpha+1)} {}_2F_1\left(-\sigma_r,1;\alpha+1;-\frac{1}{k}\right) - h^\alpha \sum_{j=0}^p \beta_j^{(p)} t_{k+1-j}^{\sigma_r} = \mathcal{O}(h^{\alpha+p+1} t_{k+1}^{\sigma_r-p-1}),
\end{align*}
which is equivalent to
\begin{align*}
\frac{k^{\sigma_r}}{\Gamma(\alpha+1)} {}_2F_1\left(-\sigma_r,1;\alpha+1;-\frac{1}{k}\right) - \sum_{j=0}^p \beta_j^{(p)} (k+1-j)^{\sigma_r} = \mathcal{O}((k+1)^{\sigma_r-p-1}).
\end{align*}
Hence, by \eqref{2.20}, we have
\begin{align*}
\sum_{j=1}^{m_u} W_{k,j}^{(\alpha,\sigma,p)} j^{\sigma_r} = \mathcal{O}((k+1)^{\sigma_r-p-1}),~~~r=1,2,\ldots,m_u.
\end{align*}
Similarly, by using Lemma \ref{lemma2.1} and \ref{lemma2.2}, we can obtain
\[\sum_{j=1}^{m_f} W_{k,j}^{(\alpha,\delta,p)} j^{\delta_r} = \mathcal{O}((k+1)^{\delta_r-p-1}),~~~r=1,2,\ldots,m_f,\]
\[\sum_{j=1}^{\tilde{m}_u} \tilde{W}_{k,j}^{(\alpha,\sigma,p)} j^{\sigma_r} = \mathcal{O}((k + 1)^{-\alpha-1}) + \mathcal{O}((k+1)^{\sigma_r-p-1}),~~~r=1,2,\ldots,\tilde{m}_u.\]
Moreover, for $W_{k,j}^{(\delta,p)}$ in \eqref{3.9}, we can get that
\[\sum_{j=1}^{\tilde{m}_f} W_{k,j}^{(\delta,0)} j^{\delta_r} = (k+1)^{\delta_r} - k^{\delta_r} = \mathcal{O}((k+1)^{\delta_r - 1}),~~~r=1,2,\ldots,\tilde{m}_f,\]
\[\sum_{j=1}^{\tilde{m}_f} W_{k,j}^{(\delta,1)} j^{\delta_r} = (k+1)^{\delta_r} - 2k^{\delta_r} + (k-1)^{\delta_r} = \mathcal{O}((k+1)^{\delta_r - 2}),~~~r=1,2,\ldots,\tilde{m}_f,\]
which ends the proof.
\end{proof}

\subsection{Proof of Theorem \ref{theorem:convergence}}

\begin{proof}
Let $e_{k+1} = u(t_{k+1}) - u_{k+1}$. When $p=0$, subtracting \eqref{3.13} from \eqref{3.14} and using the Lipschitz condition \eqref{3.2} yield
\begin{align}\label{3.16}
\|e_{k+1}\|_{\infty} \le & \| e_{k}\|_{\infty} + |\lambda| h^\alpha \bigg( \beta_0^{(0)} \|e_{k+1}\|_\infty + \sum_{j=1}^{m_u} \left|W_{k,j}^{(\alpha,\sigma,0)}\right| \|e_{j} \|_{\infty} \bigg)\nonumber\\
& - \frac{1}{\Gamma(\alpha) \Gamma(2-\alpha)} \sum_{j=0}^{k} \left|\gamma_{k,j}^{(0)}\right| \left\|e_j\right\|_\infty - \sum_{j=1}^{\tilde{m}_u} \left|\tilde{W}_{k,j}^{(\alpha,\sigma,0)}\right| \|e_{j}\|_{\infty} \nonumber\\
& + L \beta_0^{(0)} h^\alpha \bigg(\|e_{k}\|_\infty + \sum_{j=1}^{\tilde{m}_f} \left|W_{k,j}^{(\delta,0)}\right| \|e_j\|_{\infty}\bigg) \nonumber\\
& + L h^\alpha \sum_{j=1}^{m_f} \left|W_{k,j}^{(\alpha,\delta,0)}\right| \|e_j\|_\infty + R_{k+1} \nonumber\\
\le & |\lambda|\beta_0^{(0)} h^\alpha \| e_{k+1} \|_{\infty} + \| e_{k}\|_{\infty} + L \beta_0^{(0)} h^\alpha \|e_{k}\|_\infty + \sum_{j=1}^M \tilde{W}_{k,j} \|e_{j}\|_{\infty} \nonumber\\
& - \frac{1}{\Gamma(\alpha) \Gamma(2-\alpha)} \sum_{j=0}^{k} \left|\gamma_{k,j}^{(0)}\right| \left\|e_j\right\|_\infty + R_{k+1},
\end{align}
where
\[\tilde{W}_{k,j} = |\lambda| h^\alpha \left|W_{k,j}^{(\alpha,\sigma,0)}\right| + L h^\alpha \left|W_{k,j}^{(\alpha,\delta,0)}\right| + \left|\tilde{W}_{k,j}^{(\alpha,\sigma,0)}\right| + L \beta_0^{(0)} h^\alpha \left|W_{k,j}^{(\delta,0)}\right|, \]
and
\[R_{k+1} = R_{k+1}^u + R_{k+1}^f + \tilde{R}_{k+1}^u + h^\alpha \beta_0^{(p)} \tilde{R}_{k+1}^f \le v_9 h^q_0,\] 
with $q_0 = \min\left\{1, \sigma_{\tilde{m}_u+1}, \sigma_{m_u+1}+\alpha, \delta_{m_f+1}+\alpha, \delta_{\tilde{m}_f+1}+\alpha \right\}$ and $v_9 > 0$ is a constant independent of $h$. It follows from Lemma \ref{lemma2.4} there exists a constant $v_{10} > 0$ such that 
\[\tilde{W}_{k,j} \le v_{10},~~~\mbox{when}~\sigma_{\tilde{m}_u} \le 1,~\sigma_{m_u},\delta_{m_f},\delta_{\tilde{m}_f} \le \alpha+1.\]
We rewrite \eqref{3.16} as
\begin{align*}
(1-|\lambda|\beta_0^{(0)} h^\alpha) \| e_{k+1} \|_{\infty} \le & (1 + L \beta_0^{(0)} h^\alpha) \|e_{k}\|_\infty + \sum_{j=1}^M \tilde{W}_{k,j} \|e_{j}\|_{\infty} \nonumber\\
& - \frac{1}{\Gamma(\alpha) \Gamma(2-\alpha)} \sum_{j=0}^{k} \left|\gamma_{k,j}^{(0)}\right| \left\|e_j\right\|_\infty + R_{k+1}.
\end{align*}
Since by using Lemma \ref{lemma2.3} we have
\[\sum_{j=0}^{k} \left|\gamma_{k,j}^{(0)} \right| = 2 \gamma_{k,k}^{(0)} \le 2C_1.\]
Then, when $|\lambda|\beta_0^{(p)} h^\alpha < 1$, we can obtain that
\begin{align*}
\| e_{k+1} \|_{\infty} \le & \frac{1}{1\!-\!|\lambda|\beta_0^{(0)} h^\alpha} \exp\left[1 \!+\! L \beta_0^{(0)} h^\alpha \!-\! \frac{2C_1}{\Gamma(\alpha) \Gamma(2\!-\!\alpha)}\right]\! \left(\sum_{j=1}^M \tilde{W}_{k,j} \|e_{j}\|_{\infty} \!+\! R_{k+1} \right) \nonumber\\
\le & v_{11} \left(\sum_{j=1}^M \tilde{W}_{k,j} \|e_{j}\|_{\infty} + R_{k+1} \right) \le v_{11} \left( v_9 \sum_{j=1}^M \|e_{j}\|_{\infty} + v_{10} h^{q_0} \right),
\end{align*}
where the discrete Gronwall inequality in \cite{Heywood1990} has been used. For $p=1$, we can also obtain that
\begin{align*}
\| e_{k+1} \|_{\infty} \le v_{12} \left(\sum_{j=1}^M \|e_{j}\|_{\infty} + h^{q_1} \right),
\end{align*}
with $q_1 = \min\left\{2, \sigma_{\tilde{m}_u+1} + 1, \sigma_{m_u+1}+\alpha + 1, \delta_{m_f+1}+\alpha + 1, \delta_{\tilde{m}_f+1}+\alpha + 1 \right\}$ and $v_{12}$ is a positive constant independent of $h$. Therefore, this completes the proof.
\end{proof}



%

%
%



%
%
%

\bibliographystyle{siamplain}
\bibliography{SC_references}

\end{document}